\documentclass[final]{siamltex}
\usepackage{url}
\usepackage[english]{babel}
\usepackage{amsmath,amssymb} 
\usepackage{graphicx}
\usepackage{algorithm,algorithmic}
\usepackage{multirow}

\title{HOID: Higher Order Interpolatory Decomposition for tensors based on Tucker representation}

\author{Arvind K. Saibaba\thanks{Department of Mathematics, North Carolina State University $\text{asaibab}@\text{ncsu.edu}$}}

\newcommand{\C}{\mathbb{C}}
\newcommand{\bigO}{\mathcal{O}}
\newcommand{\mode}[1]{\times_{#1}}
\newcommand{\ten}[1]{\mathbf{\mathcal{#1}}}
\newcommand{\define}{\equiv} 
\newcommand{\normtwo}[1]{\| #1\|_2}
\newcommand{\normf}[1]{\| #1\|_{F}}
\newcommand{\norm}[2]{\| #1\|_{#2}}
\newcommand{\mat}[1]{\mathbf{#1}}
\renewcommand{\vec}[1]{\mathbf{#1}}
\newcommand{\argmin}{\arg\min}

\newtheorem{remark}{Remark}

\begin{document}

\maketitle
\begin{abstract}We derive a CUR-type factorization for tensors in the Tucker format based on interpolatory decomposition, which we will denote as Higher Order Interpolatory Decomposition (HOID). Given a tensor $\ten{X}$, the algorithm provides a set of column vectors $\{ \mat{C}_n\}_{n=1}^d$ which are columns extracted from the mode-$n$ tensor unfolding, along with a core tensor $\ten{G}$ and together, they satisfy some error bounds. Compared to the Higher Order SVD (HOSVD) algorithm, the HOID provides a decomposition that preserves certain important features of the original tensor such as sparsity, non-negativity, integer values, etc. Error bounds along with detailed estimates of computational costs are provided. The algorithms proposed in this paper have been validated against carefully chosen numerical examples which highlight the favorable properties of the algorithms. Related methods for subset selection proposed for matrix CUR decomposition, such as Discrete Empirical Interpolation method (DEIM) and leverage score sampling, have also been extended to tensors and are compared against our proposed algorithms. 
\end{abstract}

\section{Introduction and motivation}
Tensors, defined to be multidimensional arrays, are extremely useful in many applications ranging from neuroscience, facial recognition, and uncertainty quantification. An excellent review of the properties and applications of tensors is available in~\cite{kolda2009tensor}. A first-order tensor is a vector, a second-order tensor is a matrix, and tensors of order three or higher are called higher-order tensors. This paper presents a new CUR-type tensor decomposition for higher-order tensors based upon the interpolatory decomposition (ID) for matrices. 

Recently, a CUR decomposition has been proposed for matrices such that for a given matrix $\mat{A} \in \mathbb{C}^{m\times n}$
\[ \mat{A} = \mat{CUR} + \mat{E} \qquad \norm{\mat{E}}{} \ll \norm{\mat{A}}{}, \]
where $\mat{C} \in \mathbb{C}^{m\times k}$ is formed from columns of $\mat{A}$, $\mat{R} \in \mathbb{C}^{k\times n}$ consists of rows extracted from $\mat{A}$. The matrix $\mat{U}$ can be computed as $\mat{U} \define \mat{C}^\dagger\mat{A}\mat{R}^\dagger$ (where ${}^\dagger$ refers to the Moore-Penrose pseudo-inverse); however there are other possible choices for $\mat{U}$ as well. It is widely known that  the best rank-$k$ approximation to the matrix (in the spectral and Frobenius norm) is obtained by retaining only the top $k$ singular values and singular vectors. However, the advantage of the CUR decomposition is that the matrices $\mat{C}$ and $\mat{R}$ are representative of the matrix $\mat{A}$ itself, i.e., they preserve sparsity, non-negativity, integer values, etc.

A CUR-type decomposition for  tensor valued data was proposed by~\cite{drineas2007randomized}, in which a given tensor $\ten{X} \in \mathbb{C}^{I_1 \times \dots \times I_d}$ is approximated by a rank-$(r_1,\dots,r_d)$ tensor, obtained as the product of a core tensor  $\ten{G}\in \mathbb{C}^{r_1\times \dots \times r_d}$, and a list of matrices $\mat{C}_n \in \mathbb{C}^{I_n \times r_n}$ for $n=1,\dots d$ which are columns extracted from the mode-$n$ unfolding. The necessary background for tensors is reviewed in Section~\ref{sec:ten}. In~\cite{drineas2007randomized}, the column vectors, collected in the matrices $\mat{C}_n$, were obtained by sampling from the mode-$n$ unfolding, depending on a probability distribution based on the column norms. Our approach is different and extends the interpolatory decomposition, previously developed for matrices, to produce a  CUR-type decomposition based on Tucker format. We review previous work and highlight our main contributions here.

 \textbf{Related Work:} Previous work has considered developing matrix CUR decompositions by sampling the columns and rows of the matrix corresponding to a probability  distribution~\cite{drineas2006fast,bodor2012rcur,wang2013improving,mahoney2009cur,boutsidis2014optimal,drineas2008relative}.  Earlier methods used a sampling strategy based on a probability distribution that uses the column norms~\cite{drineas2006fast};  samples from this probability distribution were then used to extract optimal columns $\mat{C}$ and a similar approach was used to extract rows $\mat{R}$.  Drineas and Mahoney~\cite{drineas2007randomized}  applied this result to tensors  to produce low-rank representations based on the Tucker format. A related $(2+1)$ decomposition was also proposed in~\cite{mahoney2008tensor}. Sampling based on column norm sampling have been largely superseded by other sampling methods, such as those based on leverage score sampling; since they provide better relative error bounds~\cite{mahoney2009cur,boutsidis2014optimal}. In the leverage score approach, a singular value decomposition is computed, either exactly or approximately, and the leverage scores (defined as the squared two norm of the row of singular vector matrix) are used to define a probability distribution which is then used for sampling appropriate columns and rows from $\mat{A}$. To the best of our knowledge, leverage score sampling has not yet been applied to tensors.

More recently, two new methods for computing matrix CUR emerged which is relevant to our work; Sorensen and Embree~\cite{sorensen2014deim} used the Discrete Empirical Interpolation Method (DEIM)  previously developed in the context of model reduction, while Voronin and Martinsson~\cite{voronin2014cur} employed a two-sided interpolatory decomposition. Our work proposes a new algorithm based on strong rank-revealing QR but also provides an extension of the leverage score, DEIM and ID approach to tensors (see Appendices~\ref{sec:deim} and~\ref{sec:lev}).

 We would like to mention other relevant approaches here. Cross-approximation and pseudo-skeleton algorithms, which are based on a greedy approach at subset selection, have been developed to produce decompositions with interpolatory properties for the $\mathcal{H}$-Tucker format~\cite{espig2009black} and the Tensor train format~\cite{oseledets2010tt}. A review of various low-rank techniques for higher dimensional tensors is provided in~\cite{grasedyck2013literature}. A  different tensor interpolatory decomposition was proposed in~\cite{biagioni2015randomized} based on the CANDECOMP/PARAFAC  decomposition (see, for example~\cite{kolda2009tensor}); furthermore, the properties of the interpolatory decomposition proposed here are quite different compared to that paper and therefore, we will not discuss it further. Yet another CUR-type decomposition for tensors was proposed by Friedland et al~\cite{friedland2011fast}; however, it is not based on the Interpolatory Decomposition (ID).

\textbf{Contributions:} Our contributions are three fold: 1) We provide a new interpolatory decomposition for tensors in the Tucker format, of the form in~\eqref{eqn:hoid}, based on strong  rank revealing QR factorization (RRQR) applied to each mode unfolding; 2) given an approximate low-rank factorization  of the tensor we show how to compute an equivalent HOID representation; 3)  we provide a sequentially truncated HOID algorithm which is cheaper to implement than the standard HOID (in point 1), and extends the approach of~\cite{vannieuwenhoven2012new}. A detailed analysis of computational cost and the approximation error is presented. As was previously mentioned, we also provide an extension of the leverage score, DEIM and ID approach to tensors (see Appendices~\ref{sec:deim} and~\ref{sec:lev}). The theoretical bounds suggest that  the methods using strong RRQR is much better; however numerical examples will demonstrate that the error incurred using our algorithms is comparable to the methods based on DEIM and better than the leverage score approach. Our analysis is also relevant for the matrix CUR decomposition, which is a special case of the tensor algorithm that we  propose. The resulting codes are implemented in Python  using \verb|scipy.linalg.interpolative| and \verb|scikit.tensor| packages and are available on github~\url{https://github.com/arvindks/tensorcur}.

Although our main contribution is the use of strong RRQR to obtain a  tensor HOID, we improve the results of~\cite{drineas2007randomized} in several ways: 1) We provide a tighter bound than the analysis of~\cite[Theorem 2]{drineas2007randomized} for sampling based on column norms. This is because we use a sharper result of the error in Frobenius norm using the properties of orthogonal projectors (see Lemma~\ref{lemma:proj}), rather than using the triangle inequality~\cite{drineas2007randomized}. 2) We use better sampling strategies for choosing optimal columns of the mode-n unfolding $\mat{X}_{(n)}$ based on leverage scores computed from the right singular vectors of $\mat{X}_{(n)}$. The sampling using leverage scores allows us to obtain relative error estimates, which are better compared to~\cite{drineas2007randomized}\footnote{The authors in~\cite{drineas2007randomized} also note that since the time of initial submission, significant advances have been made that obtain relative error guarantees, as opposed to additive guarantees.}. These contributions have been discussed in Appendix~\ref{sec:lev}.

We conclude the introduction with a summary of this paper. In Section~\ref{sec:math}, we recall the relevant mathematical preliminaries by providing a quick background on tensors and interpolatory decompositions based on rank-revealing factorizations. Next, in Section~\ref{sec:hoid} we present our algorithms for computing a HOID based on Tucker factorization using the matrix interpolatory decomposition. We show how to derive this using both the explicit elements of the tensor, and ways to convert an already available low-rank representation. Error bounds and computational costs are also presented in detail. Finally, numerical experiments validating the bounds are presented in Section~\ref{sec:res} which demonstrates the favorable properties of the algorithms proposed.

\section{Mathematical preliminaries}\label{sec:math}
\subsection{Background on tensors}\label{sec:ten}
Here we review the basic notations and concepts involving tensors which will be useful in our discussions. A more detailed discussion of the properties and applications can be found in~\cite{kolda2009tensor}. A tensor is a $d$-dimensional array of numbers denoted by script notation $\ten{X} \in \mathbb{C}^{I_1 \times \dots \times I_d}$ with entries given by 
\[ x_{j_1,\dots,j_d} \qquad \forall \quad 1\leq j_1 \leq I_1, \> \dots, \>  1 \leq j_d \leq I_d.\]
We will denote by the matrix $\mat{X}_{(n)} \in \mathbb{C}^{I_n \times (\prod_{j\neq n} I_j)}$ the n-th mode unfolding of the tensor $\ten{X}$. Since there are $d$ dimensions, there are all together $d$-possibilities for unfolding called matricization. The $n-$mode multiplication of a tensor $\ten{X} \in \mathbb{C}^{I_1 \times \dots \times I_d}$ with a matrix $\mat{U} \in \mathbb{C}^{k\times I_n}$ results in a tensor $\ten{Y}$ of dimensions $\ten{Y} \in \mathbb{C}^{I_1 \times  \dots\times I_{n-1} \times k \times I_{n+1} \times \dots \times I_d}$ such that 
\[ \ten{Y}_{j_1,\dots,j_{n-1},j,j_{n+1} \dots,j_d}  = \left( \ten{X} \mode{n} \mat{U}\right)_{j_1,\dots,j_{n-1},j,j_{n+1} \dots,j_d} \> = \> \sum_{j_n=1}^{I_n}x_{j_1,\dots,j_d} u_{j,j_n}. \] 
Alternatively it can be expressed conveniently  in terms of matrix unfolding as 
\[ \ten{Y} = \ten{X}\mode{n}  \mat{U}  \quad \Leftrightarrow \quad \mat{Y}_{(n)} = \mat{U}\mat{X}_{(n)}.\]
Given the definitions of mode products and matricization of tensors, we can define the Higher Order SVD (HOSVD) algorithm for producing a rank $(r_1,\dots,r_d)$ approximation to the tensor based on the Tucker format. The HOSVD algorithm~\cite{de2000multilinear} returns a core tensor $\ten{G} \in \mathbb{C}^{r_1\times\dots\times r_d}$ and a set of unitary matrices $\mat{U}_j\in\mathbb{C}^{I_j\times r_j}$ for $j=1,\dots,d$ such that 
\begin{equation}\label{eqn:tucker}
\ten{X} \> \approx \> \ten{G} \mode{1} \mat{U}_1 \dots \mode{d} \mat{U}_d.
\end{equation}•
As mentioned earlier,~\eqref{eqn:tucker} is called the Tucker representation. However, a straightforward generalization to third and higher order tensors of the matrix Eckart-Young theorem~\cite{de2000best} is not possible; in fact, the best low-rank approximation is an ill-posed problem~\cite{desilva2008tensor}.  Another popular representation of tensors is called the CANDECOMP/PARAFAC decomposition (henceforth, called the CP decomposition) and is based on a sum of outer rank$-1$ products (similar to the SVD)
\begin{equation}\label{eqn:cp}
\mat{X} \> \approx \> \sum_{k=1}^r \lambda_k \vec{u}_1^{(k)} \circ \vec{u}_2^{(k)} \dots \circ \vec{u}_d^{(k)},
\end{equation}•
\noindent where the factors $\vec{u}_n \in \mathbb{R}^{I_n}$ for $n=1,\dots,d$, the factors $\lambda_n \in \mathbb{R}$ and $r$ is a positive integer. However, unlike the matrix SVD, $\vec{u}_n$ need not be orthogonal. 
\begin{algorithm}
\begin{algorithmic}[1]
\REQUIRE Tensor $\ten{X} \in \mathbb{C}^{I_1 \times \dots \times I_d}$ and desired rank $(r_1,\dots,r_d)$.
\FOR {n=1,\dots,d}
\STATE Compute $r_n$ left singular vectors $\mat{U}_n\in \mathbb{C}^{I_n\times r_n}$ of unfolding $\mat{X}_{(n)}$ .
\ENDFOR
\STATE Compute core tensor $\ten{G} \in \mathbb{C}^{r_1\times \dots \times r_d}$ as 
\[ \ten{G} \> \define \>  \ten{X} \mode{1} \mat{U}_1^* \mode{2} \dots \mode{d}\mat{U}_d^*.\]
\RETURN Tucker decomposition $[\ten{G}; \mat{U}_1,\dots,\mat{U}_d]$.
\end{algorithmic}
\caption{Higher order SVD, see for example~\cite{kolda2009tensor}.}
\label{alg:hosvd}
\end{algorithm}

We recall the following result stated and proved in~\cite[Theorem 5.1]{vannieuwenhoven2012new}, which will be useful for our subsequent analysis. 
\begin{lemma}\label{lemma:proj}
Let $\ten{X} \in \mathbb{C}^{I_1\times\dots \times I_j}$ be a tensor and let $\{ \mat{\Pi}_j \}_{j=1}^d$ be a list of $d$ orthogonal projection matrices, where $\mat{\Pi}_j \in \mathbb{C}^{I_j\times I_j}$ then
\[\normf{\ten{X} - \ten{X} \mode{1} \mat{\Pi}_1\dots \mode{d}\mat{\Pi}_d }^2  \> \leq  \>  \sum_{j=1}^d \normf{\ten{X} - \ten{X}  \mode{j}\mat{\Pi}_j }^2. \]
\end{lemma}
The result relies on the orthogonality of the projector in the Frobenius norm~\cite{vannieuwenhoven2012new}, i.e., for any $n=1,\dots,d$ defining $\mat{\Pi}_n^\perp = \mat{I} - \mat{\Pi}_n$
\begin{equation}\label{eqn:fronorm}
 \normf{\ten{X}}^2  \>=  \>  \normf{\ten{X} \mode{n} \mat{\Pi}_n}^2 + \normf{\ten{X} \mode{n} \mat{\Pi}_n^\perp}^2.
\end{equation}•

\subsection{Interpolatory decomposition}\label{sec:id}
An interpolatory decomposition for a matrix $\mat{A} \in \mathbb{C}^{m\times n}$ is the factorization 
\[ \mat{A} \approx \mat{C}  \mat{X},  \]
where $\mat{C}$ contains the columns of the matrix $\mat{A}$ indexed by an index set denoted by $c$ of size $k$, and $ \mat{X} $ is a matrix, a subset of whose columns make up the identity matrix $\mat{I}_k$, and has entries bounded by a tolerance parameter $f$, defined shortly. An interpolatory decomposition can be computed using rank-revealing QR factorizations such as pivoted QR or strong rank-revealing QR factorization, as we will now demonstrate. This was first proposed by Stewart~\cite{stewart1999four}. Consider a rank-revealing QR factorization for $\mat{A}$ 
\begin{equation}\label{eqn:pivotedqr}
\mat{A}\mat{\Pi} \> = \> \mat{QR},
\end{equation}•
where $\mat{\Pi} $ is a permutation matrix, $\mat{Q}\in \mathbb{C}^{m\times m}$ is unitary and $\mat{R}\in \mathbb{C}^{m\times n}$ is upper triangular. The factorization is computed by using column pivoting in combination with either Gram-Schmidt, Givens rotations or Householder reflections; see also Stewart~\cite{stewart1999four} for more details. 

Starting with the Pivoted QR factorization, we can decompose the permutation matrix $\mat{\Pi} = \begin{bmatrix} \mat{\Pi}_1 & \mat{\Pi}_2\end{bmatrix}$ where $\mat{\Pi}_1$ has $k$ columns and $\mat{\Pi}_2$ has $n-k$ columns. Similarly partition $\mat{Q} = \begin{bmatrix} \mat{Q}_1 & \mat{Q}_2\end{bmatrix}$ with $\mat{Q}_1$ and $\mat{Q}_2$ having $k$ and $m-k$ columns respectively. We can then write the pivoted QR factorization as 
\begin{equation}\label{eqn:pivotedqrpart}
\mat{A} \begin{bmatrix} \mat{\Pi}_1 & \mat{\Pi}_2\end{bmatrix} \> = \>\begin{bmatrix} \mat{Q}_1 & \mat{Q}_2\end{bmatrix} \begin{bmatrix}
\mat{R}_{11} & \mat{R}_{12} \\  & \mat{R}_{22} \end{bmatrix},
\end{equation}
where $\mat{R}_{11} \in \mathbb{C}^{k\times k}$, $\mat{R}_{12} \in \mathbb{C}^{k \times (n-k)}$ and $\mat{R}_{22} \in \mathbb{C}^{(m-k) \times (n-k)}$.

\begin{equation}
\mat{C} \define \mat{A}\mat{\Pi}_1 =  \mat{Q}_1 \mat{R}_{11} \qquad \mat{A\Pi}_2 = \mat{Q}_1 \mat{R}_{12} + \mat{Q}_2 \mat{R}_{22} \approx \mat{Q}_1 \mat{R}_{12}  .
\end{equation}•
 As long as $\normtwo{\mat{R}_{22}}$ is small we can approximate $ \mat{A\Pi}_2 \approx \mat{Q}_1\mat{R}_{12}$. Eliminating $\mat{Q}_1$ using the relation $\mat{Q}_1  = \mat{C}\mat{R}_{11}^{-1}$ (assuming $\mat{R}_{11}$ is invertible) we can rewrite the equations as 

\begin{equation}\label{mat_id}
\mat{A} \approx  \mat{C}\mat{F}^* \qquad  \mat{F}^* \define \begin{bmatrix} \mat{I} & \mat{R}_{11}^{-1} \mat{R}_{12} \end{bmatrix} \mat{\Pi}^*.
\end{equation}•
The decomposition written down in~\eqref{mat_id} is known as the interpolative decomposition. The interpolative decomposition approximates the matrix $\mat{A}$ using only a few of its columns, with the advantage that it preserves some important properties of the underlying matrix $\mat{A}$ such as sparsity and non-negativity. 


Next we discuss the error in the interpolatory decomposition. It can be readily seen that 
\[ \mat{A}  =  \mat{C}\mat{F}^* + \mat{E} \qquad \mat{E}  \define \begin{bmatrix} \mat{0} & \mat{Q}_2\mat{R}_{22}\end{bmatrix}\mat{\Pi}^*.\]
Since the SVD produces the optimal rank-$k$ decomposition, it follows that  $\sigma_{k+1}(\mat{A}) \leq \normtwo{\mat{R}_{22}}$. Several rank-revealing QR factorizations have been developed that satisfy the property that $\normtwo{\mat{R}_{22} }\leq C \sigma_{k+1}(\mat{A})$, where $C$ is a constant independent of the singular values of $\mat{A}$. In particular, the Gu-Eisenstat algorithm~\cite{gu1996efficient} (with parameter $f \geq 1$) produces a QR factorization with the error bounds 
\begin{equation}
\label{eqn:rrqr}
\sigma_i(\mat{R}_{11})  \>\geq \>\frac{\sigma_i(\mat{A})}{\sqrt{1 + f^2k(n-k)} } \qquad 
 \sigma_{j} (\mat{R}_{22}) \>  \leq \>   \sqrt{1 + f^2k(n-k)} \cdot \sigma_{k+j}(\mat{A}), 
\end{equation}
for $i=1,\dots,k$ and $j = 1,\dots,n-k$. Furthermore
\begin{equation}
\label{eqn:rrqr2}
| \left(\mat{R}_{11}^{-1} \mat{R}_{12} \right)_{ij}  | \leq f .
\end{equation}
The algorithm based on Pivoted QR can consistently achieve the above error bound but has the possibility of failure in adversarial cases. One such example is the notorious Kahan matrix (see for example~\cite{gu1996efficient}).

 In this work, by using the strong RRQR factorization, we obtain sharp error bounds compared to pivoted QR factorization.  In practice, however, both the strong RRQR and pivoted QR are expensive to compute since they cost $\bigO(mn^2)$; therefore, one would ideally like to avoid computing a QR factorization of $\mat{A}$. Instead, the matrix $\mat{Y} = \mat{\Omega}\mat{A}$ is formed, where $\mat{\Omega}$ is a $(k+p) \times m$ matrix with i.i.d.\ entries sampled from standard normal distribution and $p$ is an oversampling factor.  A pivoted QR factorization is performed on $\mat{Y}$ and the first $k$ columns of the permutation matrix are used to extract columns from $\mat{X}_{(n)}$. The cost is only $\bigO(mnk + mk^2)$, which makes it more efficient and the accuracy does not decrease appreciably. This is confirmed by numerical experiments (see Section~\ref{sec:res}). Further details of the randomized algorithm is provided in~\cite{halko2011finding}. In our numerical experiments, we use an implementation that is publicly available in \verb|scipy.linalg.interpolative|.

\section{Higher Order Interpolatory decomposition }\label{sec:hoid}
\subsection{Algorithm and Error bounds}
We now present the algorithm for representing low-rank tensors $\ten{X}$ in terms of an interpolatory decomposition
\begin{equation}\label{eqn:hoid}
\ten{X} \> = \> \ten{G} \mode{1} \mat{C}_1\dots \mode{d}\mat{C}_d + \ten{E},
\end{equation}•
\noindent where $\ten{E}$ is the error in the representation and $\normf{\ten{E}}$ is small relative to $\normf{\ten{X}}$. The tensor $\ten{G}$ is the core tensor and the matrices $\mat{C}_n$ for $n=1,\dots,d$ are formed by extracting $r_n$ columns from the $n$-mode unfolding $\mat{X}_{(n)}$. Therefore, $\mat{C}_n$ contain entries from the original tensor $\ten{X}$ itself. The difference between the algorithm proposed here and the work~\cite{drineas2007randomized} is the choice of the columns $\mat{C}_n$. While the authors choose a randomized selection process to choose the columns, we use the interpolatory decomposition based on strong RRQR. The index sets of the selected columns are denoted by $\vec{p}_n$ and is obtained from the interpolatory decomposition on the unfolded matrix $\mat{X}_{(n)}$, as described in Section~\ref{sec:id} 
\[\mat{X}_{(n)}\> = \> \mat{C}_n\mat{F}_n^* + \mat{E}_n \qquad n = 1,\dots,d. \]
Given the columns $\{\mat{C}_n\}_{n=1}^d$, the core tensor is computed as 
\begin{equation}\label{eqn:core}
 \ten{G} \> =\>  \ten{X} \mode{1} \mat{C}_1^\dagger \dots \mode{d}\mat{C}_d^\dagger .
 \end{equation}
 This choice of core tensor is optimal in the Frobenius norm. To see this, consider the rank-constrained minimization problem
 \[\hat{\mat{U}} \> = \>  \argmin_{\mat{U} \in \mathcal{C}(p,q,k) } \normf{\mat{X}_{(1)}  -  \mat{C}_1 \mat{U} \left( \mat{C}_d\otimes \dots \otimes  \dots \mat{C}_2\right)^*},  \]
 where $\otimes$ represents the Kronecker product, the indices $p = r_1, q = \prod_{j > 1} r_j $  and $\mathcal{C}(p,q,k)$ is the space of all complex $p \times q$ matrices of rank $k$. Here the desired rank $k=\min\{p,q\}$. Assuming all the matrices $\{ \mat{C}_j\}_{j=1}^d$ are full rank, we can invoke the discussion in~\cite{stewart1999four}, alternatively see~\cite[Theorem 2.1]{friedland2007generalized}, to obtain the optimal solution as 
 \[ \hat{\mat{U}} = \mat{C}_1^\dagger\mat{X}_{(1)} \left( \mat{C}_d\otimes \dots \otimes   \mat{C}_2\right)^{*,\dagger}  = \mat{C}_1^\dagger\mat{X}_{(1)} \left( \mat{C}_d^\dagger \otimes \dots \otimes  \mat{C}_2^\dagger\right)^*, \]
where we have used the properties of Kronecker products. Reshaping this matrix $\mat{U}$ of size $r_1 \times \prod_{j > 1} r_j$ into a tensor of dimensions $(r_1,r_2,\dots,r_d)$ gives us the desired core tensor $\ten{G}$ as represented by~\eqref{eqn:core}. The procedure for computing the HOID has been summarized in Algorithm~\ref{alg:hoid}. In several applications, the additional step of computing the core tensor $\ten{G}$, by projecting the columns onto the original tensor $\ten{X}$, is not necessary and may be skipped. 

\begin{algorithm}
\begin{algorithmic}[1]
\REQUIRE Tensor $\ten{X} \in \mathbb{C}^{I_1 \times \dots \times I_d}$ and desired rank $(r_1,\dots,r_d)$. 
\FOR {n=1,\dots,d}
\STATE Compute an interpolatory decomposition of unfolding $\mat{X}_{(n)} \approx \mat{C}_n \mat{F}^*_n$ \\
	\COMMENT{where  $\mat{C}_n \in \mathbb{C}^{I_n \times r_n}$ are columns of $\mat{X}_{(n)}$}
\ENDFOR
\STATE Compute core tensor $\ten{G} \in \mathbb{C}^{r_1\times \dots \times r_d}$ as 
\[ \ten{G} \> \define \>  \ten{X} \mode{1} \mat{C}_1^\dagger \dots \mode{d}\mat{C}_d^\dagger.\]
\RETURN Tucker decomposition $[\ten{G}; \mat{C}_1,\dots,\mat{C}_d]$.
\end{algorithmic}
\caption{HOID: Higher order Interpolatory Decomposition }
\label{alg:hoid}
\end{algorithm}
\subsubsection{Error estimate and computational cost}
The following theorem quantifies the error in the truncated interpolatory tensor decomposition when computed using Algorithm~\ref{alg:hoid}. 
\begin{theorem}\label{thm:error1}
Let the matrices $\mat{C}_n$ for $n=1,\dots,d$ and the core tensor $\ten{G}$ be computed according to Algorithm~\ref{alg:hoid}. Then we have the following error bound
\[ \normf{\ten{E}}^2  = \normf{\ten{X} - \ten{G} \mode{1} \mat{C}_1\dots \mode{d}\mat{C}_d }^2 \>  \leq \>  \sum_{n=1}^d q_n\left(\sum_{k > r_n} \sigma_{k}^2(\mat{X}_{(n)}) \right),\]
\noindent where the factors $q_n \define \left(1 + f^2r_n(\prod_{k \neq n} I_k - r_n )\right)$ and $f \geq 1$ is the parameter chosen from strong RRQR algorithm~\cite{gu1996efficient}. 
\end{theorem}•
\begin{proof}
First plug in the definition of the core tensor $\ten{G}$
  \begin{equation*}\label{thm1_inter} \ten{G} \mode{1} \mat{C}_1\dots \mode{d}\mat{C}_d = \ten{X} \mode{1} \mat{C}_1\mat{C}_1^\dagger \mode{2} \dots \mode{d} \mat{C}_d\mat{C}_d^\dagger, \end{equation*} and observe that $\mat{C}_n\mat{C}_n^\dagger$ is a projection matrix. We can then 
invoke the result from Lemma~\ref{lemma:proj} so that
 \begin{eqnarray}\nonumber
  \normf{\ten{X} - \ten{G} \mode{1} \mat{C}_1\dots \mode{d}\mat{C}_d }^2 \>  \leq & \> \sum_{n=1}^d  \normf{\ten{X} - \ten{X} \mode{n} \mat{C}_n \mat{C}_n^\dagger}^2 \\\label{thm1_inter2}
\> = &\> \sum_{n=1}^d \normf{(\mat{I}-\mat{C}_n\mat{C}_n^\dagger)\mat{X}_{(n)}}^2. 
\end{eqnarray}
To complete the proof we have  
\[  (\mat{I}-\mat{C}_n\mat{C}_n^\dagger)\mat{X}_{(n)} =  (\mat{I}-\mat{C}_n\mat{C}_n^\dagger)(\mat{C}_n\mat{F}_n^* + \mat{E}_n ) =   (\mat{I}-\mat{C}_n\mat{C}_n^\dagger) \mat{E}_n. \]
Plugging this back into~\eqref{thm1_inter2} and using the result in~\eqref{eqn:fronorm} that an orthogonal projection matrix applied to a matrix does not increase its Frobenius norm,  the right hand side of the inequality becomes $ \sum_{n=1}^d \normf{\mat{E}_n}^2$. Next, the Frobenius norm of the error can be bounded using the result in~\eqref{eqn:rrqr}; consequently we have the desired result.
\end{proof}
\begin{remark}\label{remark1} If the singular values $\sigma_{r_n+k} (\mat{X}_n)$ decay rapidly then the singular values can be discarded and the right hand side in Theorem~\ref{thm:error1} is approximately 
\[    \sum_{n=1}^d q_n \cdot \sigma_{r_n + 1}^2(\mat{X}_{(n)}). \]
If, on the other hand,  the singular values $\sigma_{r_n+k} (\mat{X}_n)$ decay very slowly, then Theorem~\ref{thm:error1} simplifies to 
\[  \normf{\ten{E}}^2 \> \leq \>  \sum_{n=1}^d q_n  \left(\min\{I_n,\prod_{k\neq n}I_k\} - r_n\right)\cdot \sigma_{r_n + 1}^2(\mat{X}_{(n)}). \]
\end{remark}
We now discuss the computational cost of Algorithm~\ref{alg:hoid}. As was mentioned earlier, the cost of computing strong RRQR factorization for an $m\times n$ matrix scales as $\bigO(mn^2)$. Instead we use a randomized approach for column subset selection which first multiplies each matrix unfolding as $\mat{Y}_k = \mat{\Omega}_k\mat{X}_{(k)}$ where $\mat{\Omega}_k$ is a $(I_k + p) \times \prod_{k\neq n} I_n$. The strong RRQR is then applied to the matrix $\mat{Y}_k$ instead of $\mat{X}_{(k)}$.  A complete error analysis of the randomization is out of the scope of this paper. Numerical results indicate that the method is comparable in accuracy to the full HOID. The cost of the resulting algorithm has been summarized in Table~\ref{tab:cost1}. 

\begin{table}[!ht] \centering
\begin{tabular}{|c|c|c|} \hline
•Step & Description & Cost \\ \hline 
1 & Randomized ID &  $ \bigO \left( \sum_{n=1}^d r_n \prod_{k=1}^n I_k  + r_n^2 I_n\right) $ \\ \hline
2 & Core Tensor & $\bigO\left( \sum_{n=1}^d\prod_{ j \leq n } r_j \prod_{k \geq n} I_k+ r_n^2 I_n \right) $  \\ \hline
\end{tabular}•
\caption{Computational cost of the Higher Order Interpolatory Decomposition, Algorithm~\ref{alg:hoid}.}
\label{tab:cost1}
\end{table}•
\subsection{Converting an existing low-rank decomposition}
Several algorithms are available in the literature for approximate low-rank representation of tensors. Examples include HOSVD~\cite{de2000multilinear} (summarized in Algorithm~\ref{alg:hosvd}), Higher Order Orthogonal Iteration (HOOI) and Alternating Least Squares algorithm (ALS).  For details on these algorithms please refer to~\cite{kolda2009tensor}. The output of these algorithms are available either in Tucker or CP format. However, these low-rank representations all share one deficiency, namely, the low-rank representations do not provide an interpretation in terms of the entries of the tensor, and do not preserve sparsity, non-negativity, etc. 

We will now address the question of how to convert a low-rank representation, available in Tucker or CP format, to an equivalent Higher Order Interpolatory Decomposition. To be fairly general, assume that we have the following low-rank representation: The mode-$n$ unfolding of the tensor can be written as a low-rank approximation satisfies the following bound
\begin{equation}\label{eqn:approx}
\normf{\mat{X}_{(n)} - \mat{A}_n\mat{B}_n^*} \> \leq \> \varepsilon_n  \qquad n=1,\dots,d.
\end{equation}•
\subsubsection{Conversion into low-rank format}
 Here we provide two examples of how to treat low-rank factorizations obtained from other algorithms that are provided to us either in Tucker format or CP format. 
\begin{enumerate}
\item The Tucker decomposition written in short hand as $[\ten{G}; \mat{U}_1,\dots, \mat{U}_d]$ admits the matrix unfolding 

\begin{equation}
\mat{X}_{(n)} \approx \mat{U}_n  \mat{G}_{(n)} (\mat{U}_d\otimes \dots \otimes \mat{U}_{n+1} \otimes \mat{U}_{n-1} \otimes \dots \otimes \mat{U}_1 )^*,
\end{equation}•
where $ \mat{G}_{(n)}$ is the $n$-th matrix unfolding of $\ten{G}$. Several possible choices exist: we choose $\mat{A}_n = \mat{U}_{(n)}$ and $\mat{B}_{(n)} \define\mat{G}_{(n)}( \mat{U}_d\otimes \dots \otimes \mat{U}_{n+1} \otimes \mat{U}_{n-1} \otimes \dots \otimes \mat{U}_1 )^*$.
\item 
The CP decomposition is a summation of rank-1 outer products and can be expressed conveniently as
\begin{equation}
\ten{X} \approx [\mat{\Lambda}; \mat{Z}_1,\dots,\mat{Z}_d] = \sum_{r=1}^R \lambda_r \vec{z}_r^{(1)} \circ \dots \circ \vec{z}_r^{(d)}.
\end{equation}•
The mode-$n$ unfolding of the CP decomposition can be expressed in terms of the Khatri-Rao product
\begin{equation}
\mat{X}_{(n)} \approx \mat{Z}_n \mat{\Lambda} \left(\mat{Z}_d\odot \dots \odot \mat{Z}_{n+1} \odot \mat{Z}_{n-1} \odot \dots \odot \mat{Z}_1 \right)^*.
\end{equation}•
As before, several choices are possible for $\mat{A}_n$ and $\mat{B}_n$. We choose $\mat{A}_n =  \mat{Z}_n$ and $\mat{B}_n =   \mat{\Lambda} \left(\mat{Z}_d\odot \dots \odot \mat{Z}_{n+1} \odot \mat{Z}_n \odot \dots \mat{Z}_1 \right)$. An alternative approach is to use the fact that the CP decomposition may be converted into a low-rank approximation by first expressing it as a special case of a Tucker decomposition and then using the mode-n unfolding of the Tucker decomposition. To see this, we can write 
\[\ten{X} \approx [\mat{\Lambda}; \mat{Z}_1,\dots,\mat{Z}_d]  = \ten{D} \mode{1} \mat{Z}_1 \dots \mode{d}\mat{Z}_d,  \]
where $\ten{D}$ is a super-diagonal tensor with diagonal entries $\ten{D}_{i,\dots,i} = \lambda_i $ and zeros otherwise. 
\end{enumerate}•

\subsubsection{Algorithm}
Assuming a low-rank representation of the form~\eqref{eqn:approx} is available, we show how to convert into an equivalent HOID.

For convenience of notation, in the subsequent discussion, we drop the subscript $n$. The understanding is that the following steps are performed for each mode unfolding. The first step involves computing the SVD of the low-rank approximation.  Given a matrix $\mat{X}$ satisfying $\norm{\mat{X} - \mat{AB}^*}{} \leq \varepsilon$, an approximate SVD of $\mat{X}$ can be obtained by the following steps:
\begin{enumerate}
\item Compute thin QR factorization $\mat{Q}_A\mat{R}_A = \mat{A}$ and $\mat{Q}_B\mat{R}_B = \mat{B}$.
\item Form $\mat{M}=\mat{R}_A\mat{R}_B^*$ and compute its SVD $\mat{M} = {\mat{U}}_M\mat{\Sigma}{\mat{V}}_M^*$.
\item Compute $\mat{U} = \mat{Q}_A\mat{U}_M$ and $\mat{V} = \mat{Q}_B{\mat{V}}_M$.
\end{enumerate}•
Return the approximate SVD $\norm{\mat{X} - \mat{U\Sigma V}^*}{}  \leq \varepsilon$.

The next step involves extracting the relevant columns from each mode-$n$ unfolding. Given an orthonormal basis $\mat{V}$ for the column space of $\mat{X}$,  we seek a set of distinct indices $\vec{p}$ that are representative of the entries of the tensor $\ten{X}$. The indices can then be used to define an interpolatory projector onto the range space of $\mat{V}$. This is defined as follows: 
\begin{equation}\label{eqn:interp_proj}
 \mat{\Pi}_\vec{p} \> \define \> \mat{P} (\mat{V}^*\mat{P})^{-1} \mat{V}^*,\end{equation}•
\noindent provided that $(\mat{V}^*\mat{P})$ is invertible,  where the matrix $\mat{P} = \mat{I}(:,\vec{p})$ contains columns of the identity matrix. It can be readily verified that $\mat{\Pi}_\vec{p}$ is a projector, i.e., it satisfies $\mat{\Pi}_\vec{p}^2 = \mat{\Pi}_\vec{p}$; however, it is an oblique projector and not an orthogonal projector. It has the following ``interpolatory property'' that for any vector $\vec{x} \in \C^{m}$, provided $\mat{\Pi}_\vec{p} \neq \mat{0},\mat{I}$
\[ (\mat{\Pi}_\vec{p}^* \vec{x})(\vec{p})  \>=\> \mat{P}^* \mat{\Pi}_\vec{p}^*\vec{x} \> = \> \mat{P}^* \mat{V} (\mat{P}^*\mat{V})^{-1} \mat{P}^*\vec{x}\> =\> \mat{P}^*\vec{x} \>= \>\vec{x}(\vec{p}).\]
In other words, the action of the projector $\mat{\Pi}_\vec{p}^*$ only extracts the entries of the vector at indices given by $\vec{p}$~\cite{chaturantabut2010nonlinear,sorensen2014deim}. Additionally the interpolatory projector has the following property, provided $\mat{\Pi}_\vec{p} \neq \mat{0},\mat{I}$
\begin{equation}\label{eqn:interp_2} \normtwo{\mat{I} -\mat{\Pi}_\vec{p}}\>  = \> \normtwo{\mat{\Pi}_\vec{p}} \> =\> \normtwo{(\mat{V}^*\mat{P})^{-1}}. 
\end{equation}
After computing the approximate right singular vectors, we can run RRQR on $\mat{V}^*_n$ to obtain the $r_n$ indices, which can be used to extract the relevant columns collected in matrices $\mat{C}_n$.  The core tensor can then be computed using~\eqref{eqn:core}. The resulting procedure is summarized in Algorithm~\ref{alg:hoid2}. Alternative strategies using either DEIM (Algorithm~\ref{alg:deim}) or sampling based on leverage scores (Algorithm~\ref{alg:deim2}) can be used to obtain the set of column indices $\vec{p}_n$ for each mode-$n$. Section~\ref{sec:res} discusses the performance between the different subset selection procedures described in this paper.

\begin{algorithm}
\begin{algorithmic}[1]
\REQUIRE Tensor $\ten{X} \in \mathbb{C}^{I_1 \times \dots \times I_d}$ in low-rank form with ranks $(r_1,\dots,r_d)$. 
\FOR {n=1,\dots,d}
\STATE Compute the SVD of unfolding $\mat{X}_{(n)} \approx \mat{A}_n\mat{B}_n^* = \mat{U}_n\mat{\Sigma}_n \mat{V}_n^*$.
\STATE Compute an index set $\vec{p}_n \in \mathbb{N}^{r_n}$ by applying RRQR on $\mat{V}_n^*$.
\STATE Extract columns indexed by $\vec{p}_n$ from the unfolding $\mat{X}_{(n)}$ denoted by $\mat{C}_n$.
\ENDFOR
\STATE If necessary, compute core tensor $\ten{G} \in \mathbb{C}^{r_1\times \dots \times r_d}$ as
\[ \ten{G} \>\define \>  \ten{X} \mode{1} \mat{C}_1^\dagger \dots \mode{d}\mat{C}_d^\dagger.\]
\RETURN Tucker decomposition $[\ten{G}; \mat{C}_1,\dots,\mat{C}_d]$.
\end{algorithmic}
\caption{Converting an existing low-rank decomposition into HOID format }
\label{alg:hoid2}
\end{algorithm}
\subsubsection{Error estimate and computational cost}
We now derive an estimate for the error incurred to produce an interpolatory decomposition based on an approximate SVD of the mode-$n$ unfolding. We first present a result, related to~\cite[Lemma 4.2]{sorensen2014deim}, and use it to derive Theorem~\ref{thm:error2}.

\begin{lemma} \label{lemma:deim}
Assume that $\mat{P}^*\mat{V}$ is invertible and let $\mat{\Pi}_\vec{p} = \mat{P}(\mat{V}^*\mat{P})^{-1}\mat{V}^*$ be an interpolatory projector. If $\mat{V}$ is orthonormal then for any $\mat{A} \in \mathbb{C}^{m\times n}$ 
\begin{equation}
\normf{\mat{A} - \mat{A\Pi_\vec{p}} } \> \leq \> \normtwo{\mat{I} - \mat{\Pi}_\vec{p} }\normf{\mat{A}(\mat{I}-\mat{VV}^*)}{}.
\end{equation}•
\end{lemma}

\begin{proof}
The proof is adapted from  Sorensen and Embree~\cite[Lemma 4.1]{sorensen2014deim}. Since $\mat{V}^*\mat{\Pi}_\vec{p} = \mat{V}^*$, therefore $\mat{V}^*(\mat{I} - \mat{\Pi}_\vec{p}) = \mat{0}$. Therefore,  
\[ \mat{A} (\mat{I} - \mat{\Pi}_\vec{p}) = \mat{A}  (\mat{I} - \mat{V}\mat{V}^*)(\mat{I} - \mat{\Pi}_\vec{p}) .\]
Taking the Frobenius norm of $\mat{A} (\mat{I} - \mat{\Pi}_\vec{p})$, and applying the sub-multiplicative property of the Frobenius norm, the result follows.
\end{proof}


\begin{theorem}\label{thm:error2}
Let the matrices $\mat{C}_n$ for $n=1,\dots,d$ and the core tensor $\ten{G}$ be computed according to Algorithm~\ref{alg:hoid}. Then we have the following error bound
\begin{equation}\label{eqn:errbound2}
 \normf{\ten{E}}^2 \> = \>  \normf{\ten{X} - \ten{G} \mode{1} \mat{C}_1\dots \mode{d}\mat{C}_d }^2  \> \leq \> \sum_{n=1}^d q_n \cdot \varepsilon_n^2,  
\end{equation}•
where the factors $q_n$ are defined in Theorem~\ref{thm:error1}, and $\varepsilon_n$ is defined in~\eqref{eqn:approx}.
\end{theorem}•
\begin{proof}
The proof has three main steps.

\paragraph{1. Exploiting structure} From Lemma~\ref{lemma:proj} 
\[\normf{\ten{X} - \ten{G} \mode{1} \mat{C}_1\dots \mode{d}\mat{C}_d }^2    \> \leq \> \sum_{n=1}^d  \normf{\ten{X} - \ten{X} \mode{n} \mat{C}_n \mat{C}_n^\dagger}^2 \> \leq \> \sum_{n=1}^d \normf{(\mat{I} - \mat{C}_n \mat{C}_n^\dagger)\mat{X}_{(n)}}^2. \]
 We define the interpolatory projector $\mat{\Pi}_n \define \mat{P}_n (\mat{V}_n^*\mat{P}_n)^{-1} \mat{V}_n^*$. We can apply the result of~\cite[Lemma 4.2]{sorensen2014deim}; however, we provide an alternative proof which is much shorter. Write 
\[ (\mat{I} - \mat{C}_n \mat{C}_n^\dagger)\mat{X}_{(n)} = (\mat{I} - \mat{C}_n \mat{C}_n^\dagger)\mat{X}_{(n)}(\mat{I}- \mat{\Pi}_n) + (\mat{I} - \mat{C}_n \mat{C}_n^\dagger)\mat{X}_{(n)}\mat{\Pi}_n.\]

Next, from $\mat{X}_{(n)}\mat{P}_n = \mat{C}_n$, follows $\mat{X}_{(n)}\mat{\Pi}_n = \mat{C}_n (\mat{V}_n^*\mat{P}_n)^{-1} \mat{V}_n^*$ and therefore, $(\mat{I} - \mat{C}_n \mat{C}_n^\dagger)\mat{X}_{(n)}\mat{\Pi}_n = \mat{0}$.

Since $(\mat{I} - \mat{C}_n \mat{C}_n^\dagger)$ is an orthogonal projector and  $ \normtwo{\mat{I} - \mat{C}_n\mat{C}_n^\dagger} \leq 1$, using the sub-multiplicativity of the Frobenius norm ,
\begin{align} \nonumber
\normf{\ten{E}}^2 \> \leq &  \> \sum_{n=1}^d \normf{(\mat{I} - \mat{C}_n \mat{C}_n^\dagger)\mat{X}_{(n)}(\mat{I}- \mat{\Pi}_n)}^2 \\ \nonumber
\leq &  \> \sum_{n=1}^d \normtwo{(\mat{I} - \mat{C}_n \mat{C}_n^\dagger)}^2\normf{\mat{X}_{(n)}(\mat{I}- \mat{\Pi}_n)}^2 \\ 
\label{eqn:inter1}
    \leq & \> \sum_{n=1}^d  \normtwo{\mat{I} - \mat{\Pi}_n}^2 \normf{\mat{X}_{(n)}(\mat{I}- \mat{V}_n\mat{V}_n^*)}^2.
\end{align}•
The last step follows from Lemma~\ref{lemma:deim}.
 \paragraph{2. Bounding $\normf{\mat{X}_{(n)}(\mat{I}- \mat{V}_n\mat{V}_n^*)}$} Observe that $\mat{U}_n\mat{\Sigma}\mat{V}_n^*(\mat{I}- \mat{V}_n\mat{V}_n^*) = 0$. Therefore, from~\eqref{eqn:approx},  we have
\begin{align} \nonumber\normf{\mat{X}_{(n)}(\mat{I}- \mat{V}_n\mat{V}_n^*)} \> =& \> \normf{(\mat{X}_{(n)} - \mat{U}_n\mat{\Sigma}_n\mat{V}_n^*) (\mat{I}- \mat{V}_n\mat{V}_n^*)}  \\ \label{eqn:inter2} \leq & \> \normf{\mat{X}_{(n)} - \mat{U}_n\mat{\Sigma}_n\mat{V}_n^*}. \end{align}
The intermediate step follows since an orthogonal projection operator does not increase the Frobenius norm. Recall that $ \normf{\mat{X}_{(n)} - \mat{U}_n\mat{\Sigma}_n\mat{V}_n^*} \leq \varepsilon_n$, where $\varepsilon_n$ is defined in~\eqref{eqn:approx}.

 \paragraph{3. Bounding $\normtwo{\mat{I} - \mat{\Pi}_n}^2$} From~\eqref{eqn:interp_2}, provided $\mat{\Pi}_\vec{p} \neq \mat{0},\mat{I}$ $$\normtwo{\mat{I} - \mat{\Pi}_n} = \normtwo{\mat{\Pi}_n} = \normtwo{(\mat{V}^*_n\mat{P}_n)^{-1}}.$$
 Next,  following~\cite{ipsen2015} we consider the term $\normtwo{(\mat{V}^*_n\mat{P}_n)^{-1}}$. From the bounds in~\eqref{eqn:rrqr} we have that 
\begin{equation}\label{eqn:rrqrinter}
\mat{V}_n^*[ \mat{P}_n, \mat{P}_n^c] = \mat{Q} [\mat{R}_{11}^{(n)},  \mat{R}_{22}^{(n)}]\quad \text{and}\quad  \sigma_{i}(\mat{R}_{11}^{(n)}) \geq\frac{\sigma_i(\mat{V}_n)}{\sqrt{1   +f^2r_n(\prod_{k\neq n} I_k - r_n )}},
\end{equation}• 
\noindent with $i=1,\dots r_n$. Since $\mat{V}_n$ is an orthonormal matrix $\sigma_i(\mat{V}_n) = 1$ for $i=1,\dots, r_n$. From this we can conclude that 
\[ \normtwo{(\mat{V}^*_n\mat{P}_n)^{-1} }  \> = \> \normtwo{(\mat{R}_{11}^{(n)})^{-1}} \> = \> \left(\sigma_{r_n}(\mat{R}_{11}^{(n)})\right)^{-1},  \]
and combining with~\eqref{eqn:rrqrinter} we obtain 
\begin{equation}\label{eqn:inter3}
\normtwo{\mat{I}- \mat{\Pi}_n}^2 \> = \> \normtwo{(\mat{V}^*_n\mat{P}_n)^{-1} }^2 \>\leq  \>  1 +   f^2r_n(\prod_{k\neq n} I_k - r_n ) \> \equiv \> q_n.
\end{equation}•

Plugging the results from~\eqref{eqn:inter2} and~\eqref{eqn:inter3} into~\eqref{eqn:inter1}, we see readily see that~\eqref{eqn:errbound2} follows. 
\end{proof}

Suppose that we use the exact singular vectors $\mat{V}_n$ corresponding to the mode-unfolding $\mat{X}_{(n)}$; say using HOSVD Algorithm~\ref{alg:hosvd}, then 
 \[ \normf{\mat{X}_{(n)}(\mat{I}- \mat{V}_n\mat{V}_n^*)} \> =  \> \left( \sum_{k > r_n} \sigma_{k}^2(\mat{X}_{(n)}) \right)^{1/2}, \]
in which case the result of Theorem~\ref{thm:error2} is sub-optimal over Theorem~\ref{thm:error1}. The comments in Remark~\ref{remark1} are also relevant here. 
 
In this paper, we assume that the exact singular values and vectors are not available, but they are computed approximately, i.e., we  assume  that the low-rank decomposition satisfies~\eqref{eqn:approx}. From the result of Theorem~\ref{thm:error2}, we can see that converting an approximate low-rank representation into a interpolatory decomposition can worsen the resulting representation error. The bound in Theorem~\ref{thm:error2} also suggests a truncation strategy: to achieve an overall tolerance $\varepsilon$, the truncation tolerance in each dimension $\epsilon_n$ must satisfy $\varepsilon_n^2 = {\varepsilon^2}/{dq_n}$. However, numerical experiments in Section~\ref{sec:res} will show that the resulting error is not significantly large and may not significantly amplify the error of the original decomposition. 

The total cost of this algorithm is summarized in Table~\ref{tab:cost2}.
\begin{table}[!ht] \centering
\begin{tabular}{|c|c|c|} \hline
•Step & Description & Cost \\ \hline 
1 & low rank SVD  &  $ \bigO \left( \sum_{n=1}^d r_n^2 (I_n + \prod_{k\neq n} I_k )  + r_n^3\right) $ \\ \hline
2 & RRQR  & $\bigO\left(\sum_{n=1}^k r_n^2  \prod_{k\neq n} I_k\right) $ \\ \hline 
3 & Core Tensor & $\bigO\left( \sum_{n=1}^d\prod_{ j \leq n } r_j \prod_{k \geq n} I_k\ +  r_n^2I_n\right) $  \\ \hline
\end{tabular}•
\caption{Computational cost of the Higher order Interpolatory Decomposition, Algorithm~\ref{alg:hoid2}.}
\label{tab:cost2}
\end{table}•

\subsection{Sequentially Truncated Higher Order Interpolatory Decomposition}
In this subsection, we present a different truncation strategy based on the Sequentially Truncated HOSVD algorithm (ST-HOSVD) proposed in~\cite{andersson1998improving} and studied by~\cite{vannieuwenhoven2012new}. As was shown by~\cite{vannieuwenhoven2012new}, this algorithm retains several of the favorable properties of truncated HOSVD algorithm while reducing the computational cost of computing the decomposition. As was summarized in Algorithm~\ref{alg:hosvd}, given a tensor $\ten{X}$ the algorithm computes the left singular vectors corresponding to the largest $r_n$ singular values of the mode-unfolding $\mat{X}_{(n)}$. Then the core tensor $\ten{G}$ is computed by multiplying $\mat{U}_n^*$ along each mode of $\ten{X}$. The HOSVD computed as described above is expensive because it involves applying the SVD algorithm on a full matrix unfolding. 

The ST-HOSVD algorithm is based on the following observation: The HOSVD algorithm can be expressed using orthogonal projectors as the following optimization problem 
\begin{align*}
\min_{\pi_1,\dots,\pi_d} \normf{\ten{X} -  \ten{X} \mode{1}\pi_1\mode{2}\cdots \mode{d}\pi_d}^2 \> = & \>
  \min_{\pi_1} \left\{ \normf{ \ten{X}\mode{1} \pi^\perp_1}^2 +  \min_{\pi_2} \left\{\normf{ \ten{X} \mode{1} \pi_1 \mode{2} \pi^\perp_2}^2  +\right.\right.\\ \label{eqn:hosvdopt}
&\>  \min_{\pi_3}     \left\{\left.\left.  + \dots +\min_{\pi_d}  \normf{\ten{X}\mode{1}\pi_1 \cdots \mode{d}\pi_d^\perp}^2 \right\}\right\}\right\}. 
\end{align*}•
The ST-HOSVD replaces this optimization problem above with a  sub-optimal optimization problem  
\begin{align*}
\min_{\pi_1,\dots,\pi_d} \normf{\ten{X} - \ten{X} \mode{1}\pi_1\mode{2}\cdots \mode{d}\pi_d}^2\quad \leq &  \quad \normf{ \ten{X}\mode{1}\hat\pi^\perp_1}^2 + \normf{ \ten{X} \mode{1} \hat\pi_1 \mode{2}\hat\pi_2^\perp}^2 + \cdots \\
& \quad  \normf{ \ten{X}  \mode{1}\hat\pi_1\mode{2}\cdots  \mode{d-1}\hat\pi_{d-1}\mode{d}\hat\pi_d^\perp}^2\,,
\end{align*}•
where the projectors $\hat\pi_n$ for $n=1,\dots,d$ are defined recursively as 
\[ \hat\pi_n = \argmin_{\pi_n} \normf{\ten{X}\mode{1}\hat\pi_1 \mode{2}\cdots  \mode{n-1}\hat\pi_{n-1}  \mode{n}\pi_n^\perp}.\]
The other important difference between ST-HOSVD compared to the truncated HOSVD is that the former algorithm is dependent on the processing order of the modes, whereas the latter algorithm is independent of the processing order. Here, for simplicity we have presented the algorithm with the processing order of the modes $\vec{p} = \{ 1,\dots, n\}$; however, the results are strongly dependent on the processing order and heuristic methods for deciding the mode order are discussed in~\cite{vannieuwenhoven2012new}.

\begin{algorithm}
\begin{algorithmic}[1]
\REQUIRE Tensor $\ten{X} \in \mathbb{C}^{I_1 \times \dots \times I_d}$ and desired rank $(r_1,\dots,r_d)$.
\STATE Define  tensor $\ten{S}^{(0)} \leftarrow \ten{X}$.
\FOR {n=1,\dots,d}
\STATE Compute the rank $r_n$ SVD of the tensor unfolding $\mat{S}_{(n)}^{(n-1)} \approx \hat{\mat{U}}_n \hat{\mat{\Sigma}}_n \hat{\mat{V}}_n^* $.
\STATE Update  $\ten{S}^{(n)} \leftarrow \hat{\mat{\Sigma}}_n \hat{\mat{V}}_n^*$.
\STATE Form unfolding $\mat{X}_{(n)}$ and its low rank approximation (only if $n>1$, otherwise skip this step and use $\tilde{\mat{V}}_n = \hat{\mat{V}}_n$)
\[ \mat{X}_{(n)} \> \approx \>  \hat{\mat{U}}_n \hat{\mat{S}}_{(n)}^{(n-1)}   \>  = \> \tilde{\mat{U}}_n\tilde{\mat{\Sigma}}_n \tilde{\mat{V}}_n^*, \]
where 
\[\hat{\ten{S}}^{(n-1)} \> \define \>{\ten{S}}^{(n)} \mode{1} \hat{\mat{U}}_1 \cdots \mode{n-1} \hat{\mat{U}}_{n-1}. \] 
\STATE Compute an index set $\vec{p}_n \in \mathbb{N}^{r_n}$ by applying RRQR on $\tilde{\mat{V}}_n^*$.
\STATE Extract columns indexed by $\vec{p}_n$ from the unfolding $\mat{X}_{(n)}$ denoted by $\mat{C}_n$.
\ENDFOR
\STATE Compute core tensor $\ten{G} \in \mathbb{C}^{r_1\times \dots \times r_d}$ as 
\[ \ten{G} \> \define \>  \ten{X} \mode{1} \mat{C}_1^\dagger \mode{2} \dots \mode{d}\mat{C}_d^\dagger.\]
\RETURN Tucker decomposition $[\ten{G}; \mat{C}_1,\dots,\mat{C}_d]$.
\end{algorithmic}
\caption{ Sequentially Truncated Higher Order Interpolatory Decomposition }
\label{alg:sthoid}
\end{algorithm}
One easy extension of ST-HOSVD to compute the HOID, is to directly apply Algorithm~\ref{alg:hoid} to the  low-rank representation from ST-HOSVD. However, since the errors in the low-rank approximation accumulate as the modes are processed, we adopt a slightly different approach.  The procedure to compute ST-HOID is summarized here. 
The procedure follows ST-HOSVD algorithm closely; instead of computing the left singular vectors $\mat{U}_n$ from the mode-$n$ unfolding of $\ten{X}$, it is approximated by the smaller tensor $\ten{S}^{(n)}$, sequentially truncated. Our modification is the following: During the ST-HOSVD procedure we explicitly compute a low-rank approximation to $\mat{X}_{(n)}$ using the pieces of information already available at step $n$. Then strong RRQR algorithm is applied to the approximate singular vectors $\tilde{\mat{V}}_n^*$ to extract the appropriate columns from $\mat{X}_{(n)}$. We also assume that the processing order of the modes have been fixed to $\vec{p} =\{1,\dots,n\}$. Algorithm~\ref{alg:sthoid} can be easily generalized to different mode orderings and alternative truncation strategies following the heuristics of~\cite{vannieuwenhoven2012new}; however, we haven't explored this in our work.

\subsubsection{Error estimate and computational cost}
We now quantify the error using ST-HOID (Algorithm~\ref{alg:sthoid}) by means of the following result.

\begin{theorem}\label{thm:error3}
Let the matrices $\mat{C}_n$ for $n=1,\dots,d$ and the core tensor $\ten{G}$ be computed according to Algorithm~\ref{alg:sthoid}. Then we have the following error bound
\[ \normf{\ten{E}}^2  = \normf{\ten{X} - \ten{G} \mode{1} \mat{C}_1\dots \mode{d}\mat{C}_d }^2 \>  \leq \>  \sum_{n=1}^d q_n\sum_{k=1}^{n} \left(\normf{\ten{X}^{(k-1)}}^2 -\normf{\ten{X}^{(k)}}^2\right),\]
\noindent where the factor $q_n$ was defined in Theorem~\ref{thm:error1} and 
\[ \ten{X}^{(n)}  \define \ten{X} \mode{1} \hat{\mat{U}}_1 \hat{\mat{U}}_1^* \dots \mode{n} \hat{\mat{U}}_n \hat{\mat{U}}_n^* \mode{n+1} \mat{I}_{I_{n+1}} \cdots \mode{d} \mat{I}_{I_d},\]
with $\ten{X}^{(0)} \define \ten{X}$.  
\end{theorem}•
\begin{proof}
Following steps 1 and 2 from the proof of Theorem~\ref{thm:error2}, the intermediate expression for the error is (i.e., combine~\eqref{eqn:inter1} and~\eqref{eqn:inter2})
\[ \normf{\ten{E}}^2 \>  \leq \> \sum_{n=1}^d \normtwo{\mat{I} - \mat{\Pi}_n}^2 \normf{\mat{X}_{(n)}- \tilde{\mat{U}}_n \tilde{\mat{\Sigma}}_n\tilde{\mat{V}}_n^*}^2, \]
where $\tilde{\mat{V}}_n$ are the approximate singular vectors constructed in Algorithm~\ref{alg:sthoid}. From step 3 of the proof of Theorem~\ref{thm:error2}, $\normtwo{\mat{I} - \mat{\Pi}_n}^2 \leq q_n$ where $q_n$ was defined in Theorem~\ref{thm:error1}. Next, we have that $\tilde{\mat{U}}_n \tilde{\mat{\Sigma}}_n\tilde{\mat{V}}_n^* =  \hat{\mat{U}}_n \hat{\mat{S}}_{(n)}^{(n-1)}$ which can be expressed in tensor form as 
\[ \hat{\mat{U}}_n \hat{\mat{S}}_{(n)}^{(n-1)} \quad \Leftrightarrow \quad \hat{\ten{S}}^{(n-1)} \mode{n}\hat{\mat{U}}_n =  {\ten{S}}^{(n)} \mode{1} \hat{\mat{U}}_1 \cdots \mode{n-1} \hat{\mat{U}}_{n-1}\mode{n} \hat{\mat{U}}_n.\] 
Since the intermediate tensors $\ten{S}^{(n)} = \ten{X} \mode{1} \hat{\mat{U}}_1 \cdots \mode{n-1} \hat{\mat{U}}_{n-1}\mode{n} \hat{\mat{U}}_n $,  $\ten{X}^{(n)} $ as defined in the statement of the theorem becomes $\ten{X}^{(n)} =\hat{\ten{S}}^{(n-1)} \mode{n}\hat{\mat{U}}_n$. Next, using the result in~\eqref{eqn:inter3} 
\[  \normf{\ten{E}}^2 \> = \>  \normf{\ten{X} - \ten{G} \mode{1} \mat{C}_1\dots \mode{d}\mat{C}_d }^2  \> \leq \> \sum_{n=1}^d q_n \normf{\ten{X} - \ten{X}^{(n)}}^2.\]
To bound $\normf{\ten{X} - \ten{X}^{(n)}}^2$ we invoke a variation of ~\cite[Theorem 6.4]{vannieuwenhoven2012new} (the proof is identical, except here the summation is restricted to mode-$n$ instead of $d$). 
\end{proof}

The cost of this algorithm is similar to Algorithm~\ref{alg:hoid2}. In addition to the costs listed in Table~\ref{tab:cost2}, the additional cost of applying the randomized ID approach to the set of matrices obtained by unfolding the tensor $\ten{S}^{(n)}$ is \[ \bigO \left( \sum_{n=1}^d\prod_{ j \leq n } r_j \prod_{k \geq n} I_k + r_n^2 I_n\right).\] 
Note that since we are computing the approximate singular pairs arising from the intermediate steps of Algorithm~\ref{alg:sthoid}, the costs for producing the low-rank decomposition is lower than the HOID applied to entire matrix (i.e., Algorithm~\ref{alg:hoid}).

\subsection{Matrix CUR decomposition}
In this section, we consider the special case of Algorithm~\ref{alg:hoid} when the dimensions of the input tensor is restricted to have dimension $2$, i.e., for matrices. 

A matrix CUR factorization to an $m\times n$ matrix produces $r$ columns and rows of $\mat{A}$, expressed as  $\mat{C} \in \mathbb{C}^{m\times r}$ and $\mat{R} \in \mathbb{C}^{r\times n}$, and an intersection matrix $\mat{U} \in \mathbb{C}^{r\times r}$ such that 
  \[ \mat{A} \>  \approx \>  \mat{CUR}. \]
There are several methods available in the literature to obtain a CUR factorization. Here we mention the following references~\cite{voronin2014cur,sorensen2014deim,wang2013improving,drineas2008relative,mahoney2009cur}; however, we emphasize that this list is, by no means, exhaustive.


The connection of the HOID with the matrix CUR decomposition is readily established by noting the following identity
\[ \mat{A} = \mat{CUR} + \mat{E} \qquad \Leftrightarrow \qquad \mat{A} = \mat{U} \mode{1} \mat{C} \mode{2} \mat{R} + \mat{E},\]
in which we have used the properties of the mode product defined in Section~\ref{sec:ten}. Here $\mat{C}$ is a matrix which represents the columns of the matrix, while $\mat{R}$ contains sampled rows from the matrix. A typical choice for the intersection matrix $\mat{U}$ is  $\mat{U} =  \mat{C}^\dagger \mat{A} \mat{R}^\dagger$ since this choice of $\mat{U}$ is the minimizer of 
$ \normf{\mat{A} - \mat{CUR}}$ in the Frobenius norm (see~\cite{stewart1999four} and ~\cite[Theorem 2.1]{friedland2007generalized}).  
With this choice of the intersection matrix, the connection with the core tensor computation follows from 
\[ \mat{U}\> =  \>  \mat{C}^\dagger \mat{A} \mat{R}^\dagger \qquad \Leftrightarrow \qquad \mat{U} = \mat{A} \mode{1} \mat{C}^\dagger\mode{2} \mat{R}^\dagger. \]
Other choice for the intersection matrices are also possible.

It can be readily seen that the results of this paper, applied to input restricted to dimension $2$, produce a matrix CUR decomposition. A special case of Theorem~\ref{thm:error1} for matrix CUR decompositions is summarized in the following corollary. This result appears to be new for matrix CUR decompositions.
\begin{corollary}
Apply Algorithm~\ref{alg:hoid} to matrix $\mat{A}\in \mathbb{C}^{m\times n}$ and rename $[\ten{G};\mat{C}_1,\mat{C}_2]$ as $[\mat{U}; \mat{C}, \mat{R}]$.  We have the following error bound
\[ \normf{\mat{A} -\mat{CUR}}^2 \> \leq q(m,n; r) \left( \sum_{k>r}\sigma_k^2\right),\]
where $q(m,n; r) = 2  + f^2r(m+n-2r)$, $f\geq 1$ is the tolerance parameter in strong RRQR~\cite{gu1996efficient}, and $\sigma_k$ are the singular values of $\mat{A}$.
\end{corollary}
\linebreak
Similarly, Theorems~\ref{thm:error2} and~\ref{thm:error3} are applicable to matrices as well. We are exploring these issues in a forthcoming paper. 
\section{Numerical Experiments}\label{sec:res}

To facilitate the comparison between different algorithms we define the following acronyms. The following algorithms act on the entries of the matrix directly to produce a low-rank approximation with rank$-(r_1,\dots,r_d)$:
\begin{enumerate}
\item HOSVD - applies the Higher Order SVD algorithm (Algorithm~\ref{alg:hosvd}) to the entries of the matrix.
\item HOID - applies the Higher Order Interpolatory Decomposition algorithm (Algorithm~\ref{alg:hoid}) and uses either the Pivoted QR labeled `HOID - PQR', or strong rank-revealing QR factorization labeled `HOID - RRQR'. The RRQR was implemented from~\cite[Algorithm 5]{gu1996efficient} with $f=1$.  
\item HORID also applies the Higher Order Interpolatory Decomposition algorithm (Algorithm~\ref{alg:hoid}). The interpolative decomposition was computed by using the package \verb|scipy.linalg.interpolative|. Details can be found in the documentation~\cite{martinsson2014id}. 
\item ST-HOID  applies the Sequentially Truncated Higher Order Interpolatory Decomposition algorithm (Algorithm~\ref{alg:sthoid})
\end{enumerate}•

 The following algorithms assume that a low-rank decomposition that has already been computed, and use this representation along with the entries of the tensor $\ten{X}$ to produce a low-rank approximation with rank$-(r_1,\dots,r_d)$:
\begin{enumerate}
\item RRQR - implements Algorithm~\ref{alg:hoid2} with column subset selection implemented using the strong RRQR algorithm~\cite[Algorithm 5]{gu1996efficient}.
\item PQR - implements Algorithm~\ref{alg:hoid2} with column subset selection implemented using Pivoted QR (see also Appendix~\ref{sec:deim}).
\item DEIM - implements Algorithm~\ref{alg:hoid2} with column subset selection implemented using DEIM Algorithm~\ref{alg:deim} summarized in Appendix~\ref{sec:deim}
\item Lev - implements Simple-Leverage - the subset selection using leverage score sampling, summarized in Appendix~\ref{sec:lev}. 
\end{enumerate}•
As a way to compare the relative performance between the above algorithms, in addition to computing the overall accuracy of the HOID representation, we also compare the error constants $\normtwo{(\mat{V}^*\mat{P})^{-1}}$ that appear in Lemma~\ref{lemma:proj} and Theorem~\ref{thm:error2}. The details of the DEIM algorithm and Leverage score based sampling have been provided in the Appendices~\ref{sec:deim} and~\ref{sec:lev}.  However, we would like to point out that the \textit{a priori} bounds obtained using RRQR are much better than DEIM, and Pivoted QR.  In practice, the results of PQR and RRQR are identical; however, the bounds for RRQR are much better and it does not experience an exponential growth for adversarial cases like the Kahan matrix~\cite{gu1996efficient}. The algorithms were all implemented in Python and all the timing results were run on an iMac desktop with 3.5 GHz i7 processor and 32 GB in memory.

\subsection{Example 1}
The following example arises from the numerical solution of integral equations.  We consider a tensor $\ten{X}$ with the entries
 \begin{equation}\label{eqn:hilb} \ten{X}_{i_1,\dots,i_d} \> = \> \frac{1}{\sqrt{i_1^2 + \dots + i_d^2} }\qquad  1\leq i_1,\dots, i_d \leq N.\end{equation}
 The above array $\ten{X}$ is obtained from a Fredholm integral equation with kernel $1/ ￼\| x- y \| $ acting on a unit hypercube and is discretized by the {N}ystr{\"o}m method on a uniform grid. The advantage of using this tensor is that an approximation in CP format exists with low-rank, i.e. has rank $r$ satisfying the estimate
 \[ r \leq C\log N \log^2 \frac{1}{\varepsilon}, \]
 where $\varepsilon$ is the relative error of the desired low-rank approximation~\cite{oseledets2008tucker}. In our examples, we choose the dimension $d=3$.

 \begin{figure}[!ht]\centering
\includegraphics[scale=0.3]{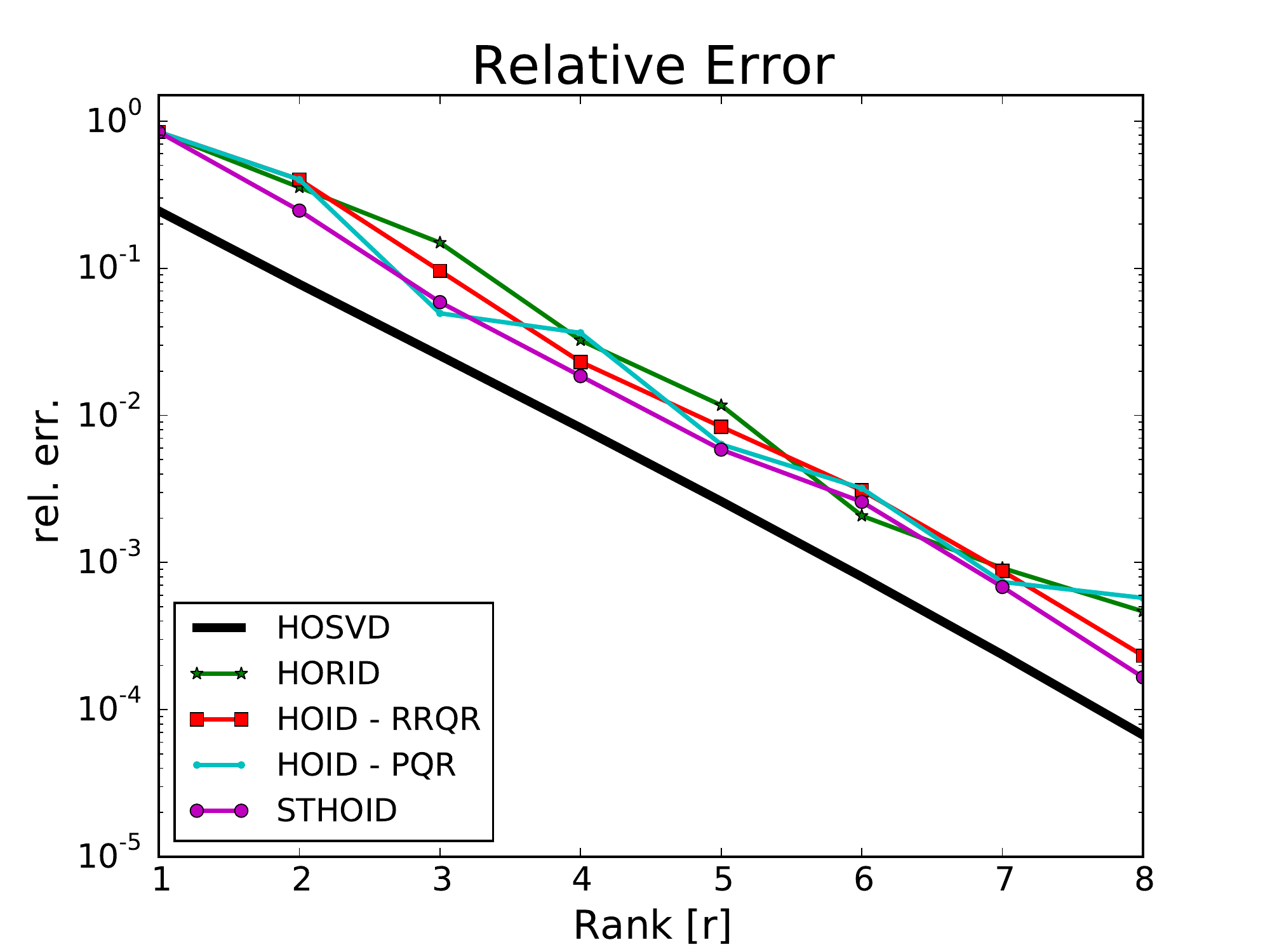}
\includegraphics[scale=0.3]{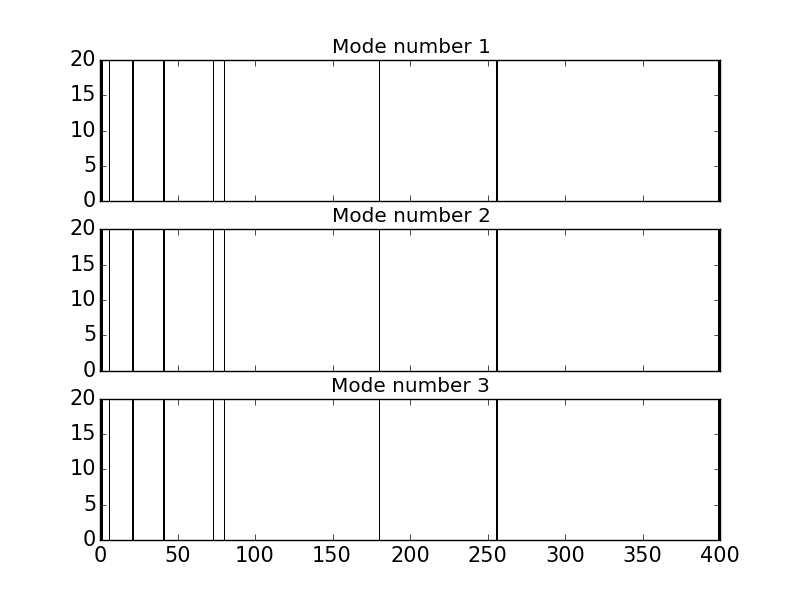}

\caption{(left) The relative error in the computation of a rank$-(r,r,r)$ approximation to the tensor defined in~\eqref{eqn:hilb}. The definitions of the algorithms used are provided at the start of Section~\ref{sec:res}. (right) The indices that has been selected using the HOID (with RRQR for subset selection) applied on the tensor $\ten{X}$ defined in~\eqref{eqn:hilb}.   }

\label{fig:relerr_hilbert}
\end{figure}

 \begin{figure}[!ht]
\centering
\includegraphics[scale=0.3]{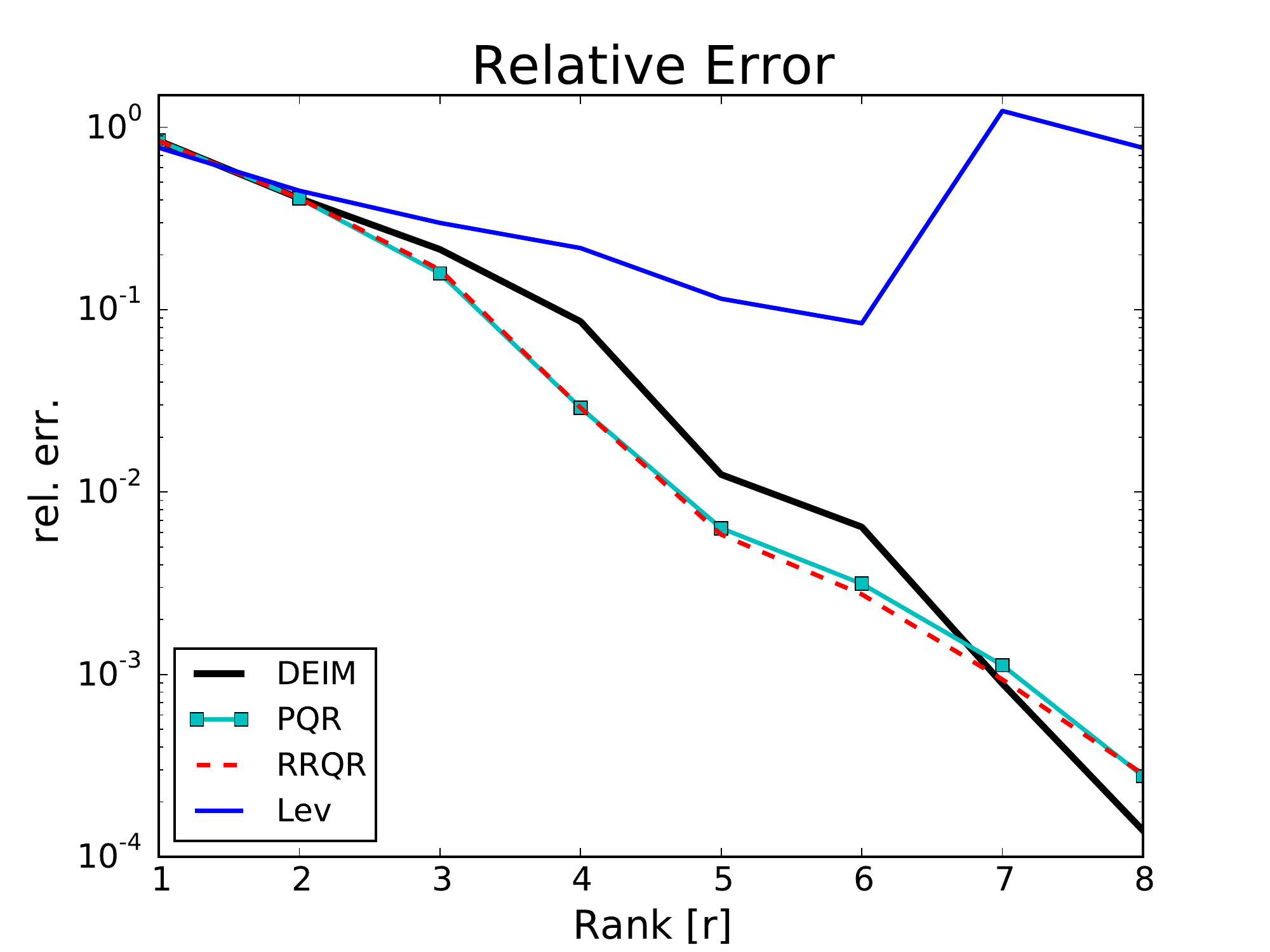}
\includegraphics[scale=0.3]{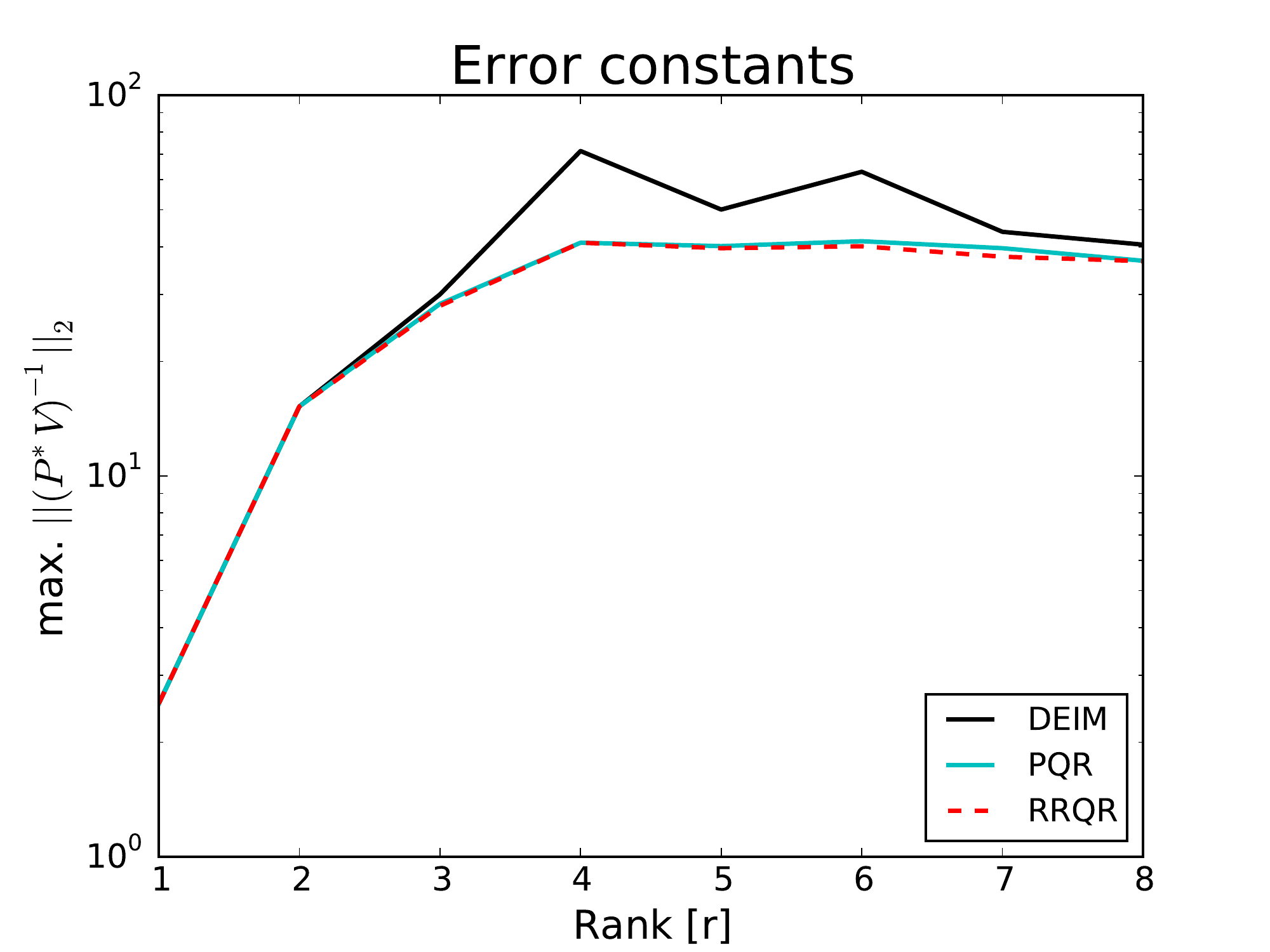}
\caption{(left) The relative error in the computation of a rank$-(r,\dots,r)$ approximation to tensor defined in~\eqref{eqn:hilb}, starting with the rank-$(r,r,r)$ approximation using  the HOSVD algorithm. The definitions of the algorithms used are provided at the start of Section~\ref{sec:res}. (right) The error constants $\max_{n=1,\dots,d}\normtwo{(\mat{P}_n^*\mat{V}_n)^{-1}}$ computed using DEIM and RRQR approaches are compared. See also Figure~\ref{fig:err_lev_hilbert} for comparison with leverage score calculations using exact singular vectors. }
\label{fig:deim_hilbert}
\end{figure}•

\begin{figure}\centering
\includegraphics[scale=0.45]{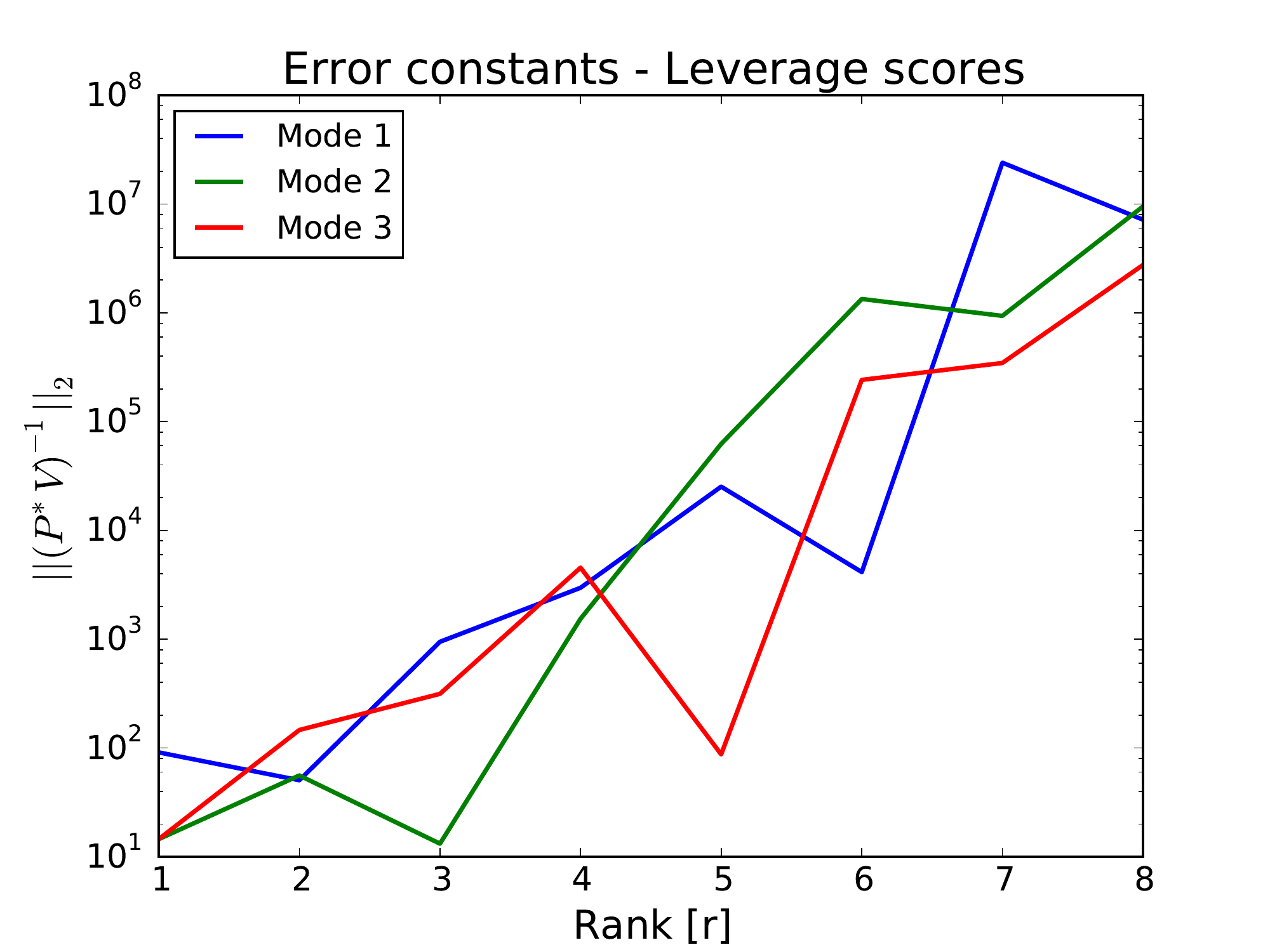}
\caption{Error constants $\| (\mat{P}^*_n\mat{V}_n)^{-1}\|_2$ for each mode computed using the leverage scores obtained from all the exact singular vectors.}
\label{fig:err_lev_hilbert}
\end{figure}

 We perform two sets of experiments on each tensor. First, we fix $N = 50$. In the first set of experiments, we compute a rank-$(r,r,r)$ decomposition using HOSVD, HOID (both PQR and RRQR),  HORID, and ST-HOID methods.  The relative error measured as $\normf{\ten{X} - \ten{X}_r}/ \normf{\ten{X}}$ is plotted as a function of the individual mode rank $r$. The results are visualized on the left part of Figure~\ref{fig:relerr_hilbert}. As can be seen the error from all the algorithms are comparable, with the error in the HOSVD algorithm being consistently lowest. In the case of matrices, SVD has the optimal accuracy for a rank-$r$ approximation; however this result is no longer the case in tensor decompositions. We also observe that the error in the HOID and HORID algorithms are similar and  only slightly higher than the HOSVD algorithm; the error is quantified by the result of Theorem~\ref{thm:error1}. Comparatively, HORID algorithm is much cheaper to compute than HOID and HOSVD. The column indices selected using the HOID algorithm have been visualized in the right part of Figure~\ref{fig:relerr_hilbert} (for visualization purpose, the indices corresponding to $N=20$ are plotted). Since all the mode$-n$ unfoldings are identical, the HOID algorithm using RRQR is deterministic and same indices are picked from each mode.  In the second set of experiments, we assume that a rank-$(r,r,r)$ decomposition has been computed using the HOSVD algorithm. We then convert it into an equivalent HOID using  DEIM, PQR and RRQR. The results are visualized in Figure~\ref{fig:deim_hilbert}. Again we observe that the three algorithms have similar performance and the error is only slightly higher compared to the HOSVD algorithm.  On the right side of the figure, we plot the error constant $\max_{n=1,\dots,d}\normtwo{(\mat{P}_n^*\mat{V}_n)^{-1}}$ as a function of the rank of the individual mode unfoldings. As can be seen, as $r$ increases, the growth in the error constants from RRQR and PQR is lower than that of DEIM but within an order of magnitude of each other. The result in the left part of Figure~\ref{fig:deim_hilbert} also shows the comparison with the Simple-Leverage method. The error in the low-rank representation seems to worsen as the requested rank-$(r,r,r)$ increases; this can be seen from the growth of the error constant $\| (\mat{P}^*_n\mat{V}_n)^{-1} \|_2$ which has been computed for each mode, see Figure~\ref{fig:err_lev_hilbert}. Even the use of exact singular vectors to compute the leverage scores does not to seem to yield an accurate HOID; in practice, note that the exact leverage scores may not be available.

\subsection{Example 2}
Following the work in~\cite{sorensen2014deim}, we construct a tensor of dimensions $\mathbb{R}^{n\times n \times n} $ in the CP format 
\begin{equation}\label{eqn:sparse}
\ten{X} =  \sum_{j=1}^{10}\frac{1000}{j} \vec{x}_j \circ \vec{y}_j \circ\vec{z}_j +  \sum_{j=11}^n\frac{1}{j} \vec{x}_j \circ \vec{y}_j \circ \vec{z}_j,
\end{equation}•
where $\vec{x}_j,\vec{y}_j,\vec{z}_j \in \mathbb{R}^{n}$ are sparse vectors with nonnegative entries.  Note that the individual vectors  are not orthonormal and there is no direct analogy with the matrix SVD in terms of a jump in the singular values. We still expect that the importance of each term in the outer product representation is decreasing with a sharp jump between terms $10$ and $11$. A HOID is relevant here because the tensor is sparse and the entries are nonnegative and we would like to preserve this structure in the column matrices $\{ \mat{C}_k\}_{k=1}^3$. 

 \begin{figure}[!ht]\centering
\includegraphics[scale=0.3]{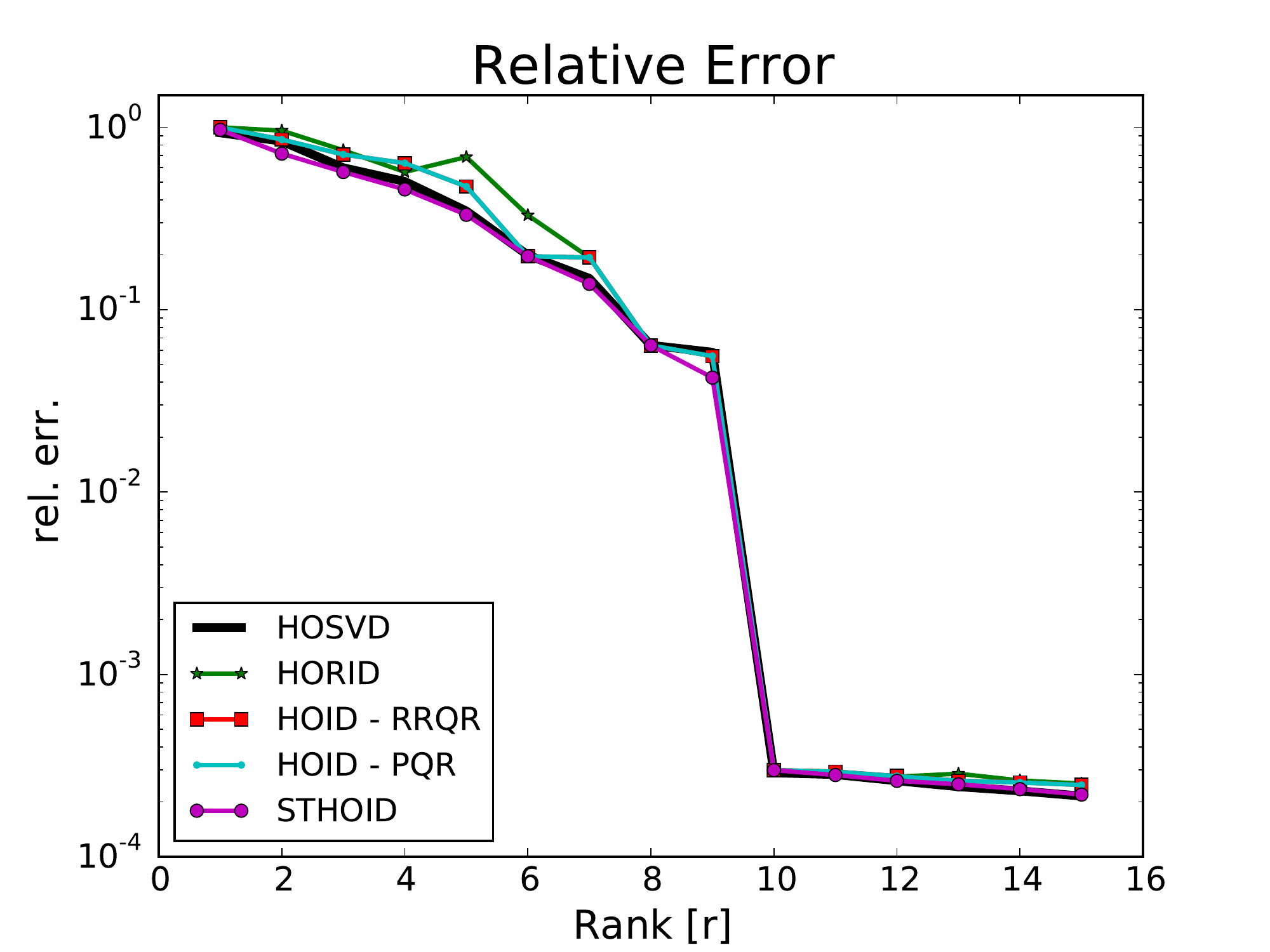}
\includegraphics[scale=0.3]{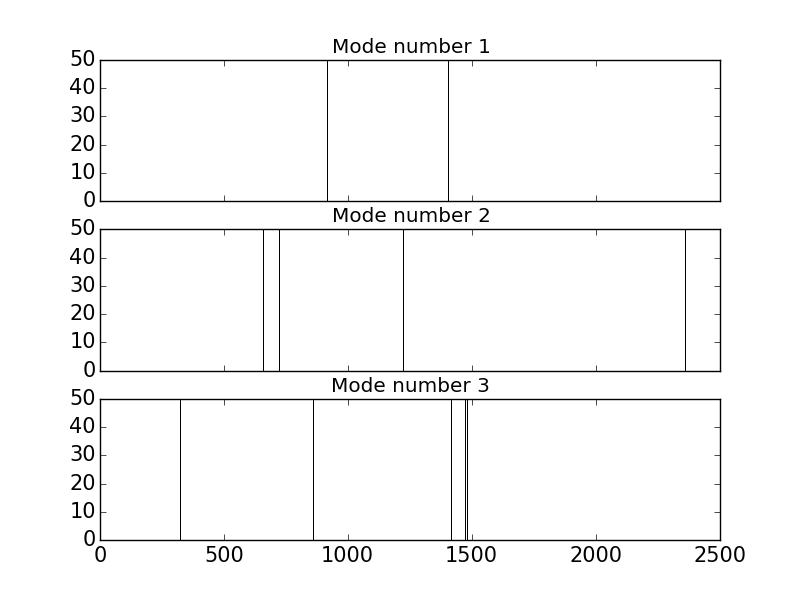}

\caption{(left) The relative error in the computation of a rank$-(r,r,r)$ approximation to the tensor defined in~\eqref{eqn:sparse}. The definitions of the algorithms used are provided at the start of Section~\ref{sec:res}. (right) The indices have been selected using the HOID (with RRQR for subset selection) applied on the tensor $\ten{X}$ defined in~\eqref{eqn:hilb}.  }
\label{fig:relerr_sparse}
\end{figure}

 \begin{figure}[!ht]
\centering
\includegraphics[scale=0.3]{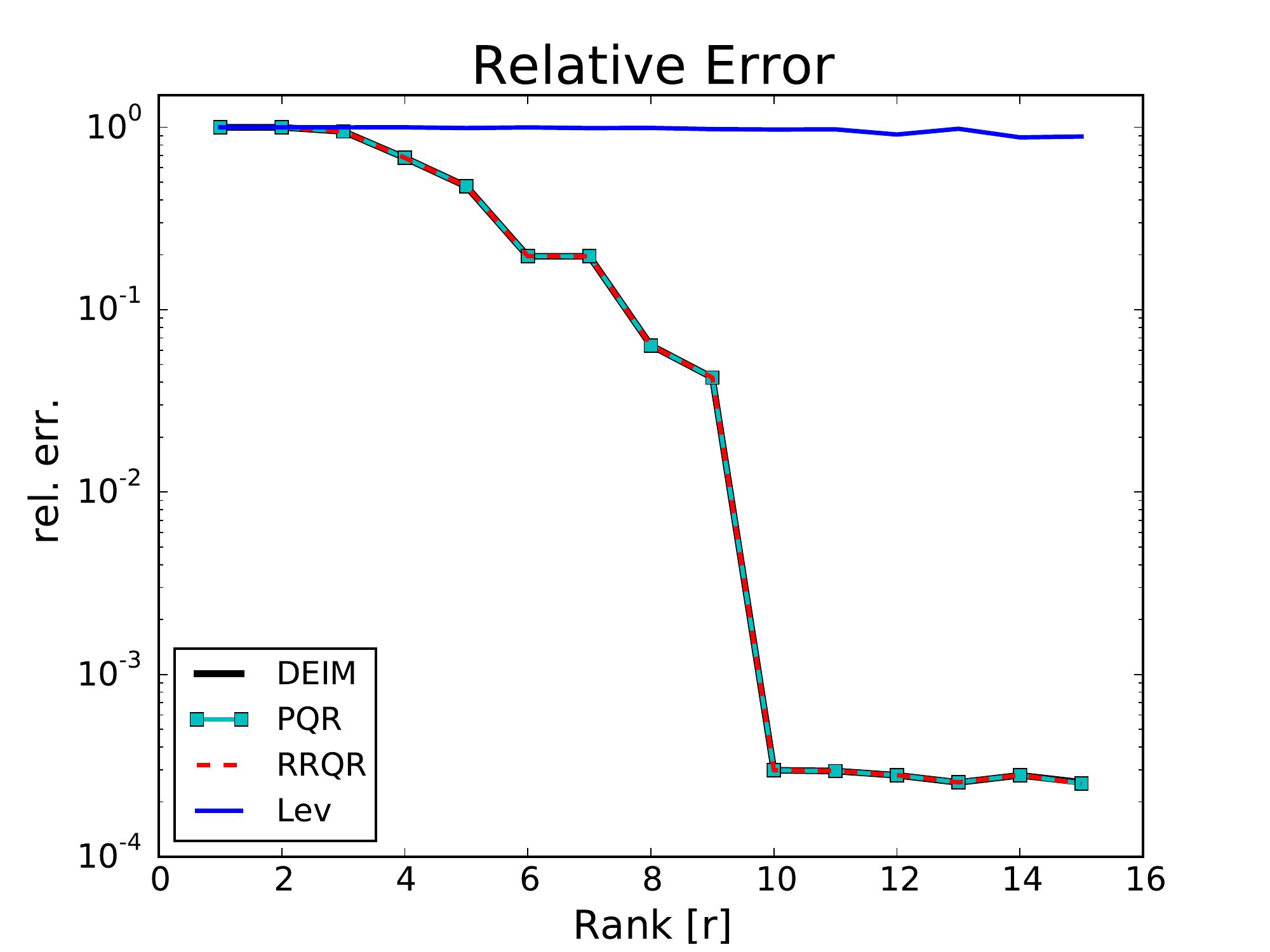}
\includegraphics[scale=0.3]{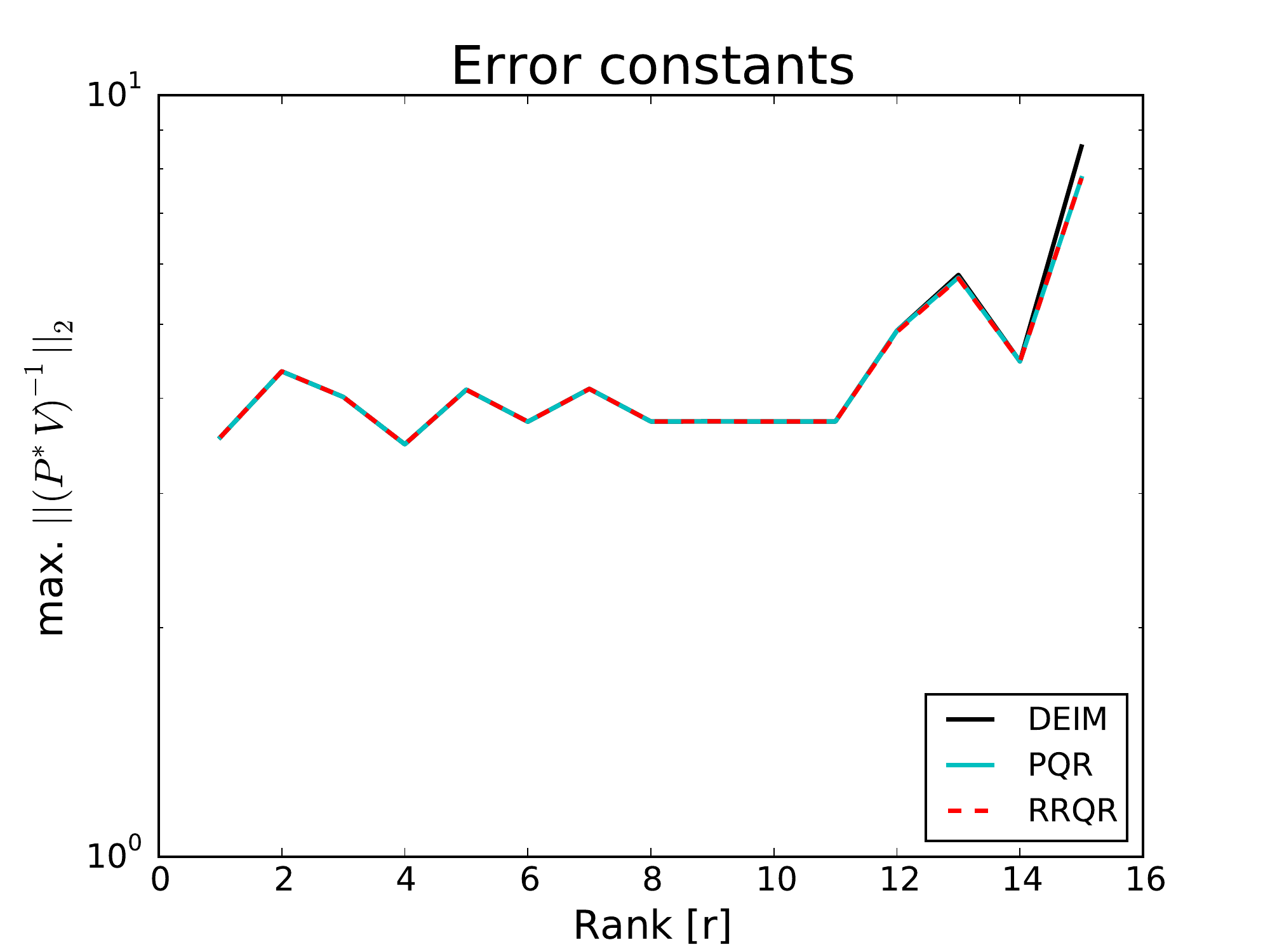}
\caption{(left) The relative error in the computation of a rank$-(r,r,r)$ approximation to tensor defined in~\eqref{eqn:sparse}, starting with the rank-$(r,r,r)$ approximation using  the HOSVD algorithm. The definitions of the algorithms used are provided at the start of Section~\ref{sec:res}. (right) The amplification factor $\max_{n=1,\dots,d}\normtwo{(\mat{P}_n^*\mat{V}_n)^{-1}}$ computed using DEIM and RRQR approach are compared. See also Figure~\ref{fig:err_lev_sparse} for comparison with leverage score calculations. }
\label{fig:deim_sparse}
\end{figure}•

We perform the same set of experiments as we did with Example 1. The results are visualized in Figures~\ref{fig:relerr_sparse},~\ref{fig:deim_sparse} and our conclusions are similar. The result in the left part of Figure~\ref{fig:deim_sparse} also shows the error of the low-rank representation using the Simple-Leverage method. As is seen from the figure, the error is quite high compared to that obtained from DEIM, PQR and RRQR. The magnitude of the error constant $\| (\mat{P}^*_n\mat{V}_n)^{-1} \|_2$, computed for each mode, is plotted in Figure~\ref{fig:err_lev_sparse}. The indices were extracted  using the leverage scores method and the leverage scores were computed using the right singular vectors (corresponding to the top $15$ singular values). For comparison, we also plot the error constants $\| (\mat{P}^*_n\mat{V}_n)^{-1} \|_2$ from the  leverage scores method;  all the right singular vectors (for each mode) were used in the computation of the leverage scores. The results of the sampling depend on the sensitivity of the leverage scores;  for a detailed discussion, see~\cite{ipsen2014sensitivity}.

\begin{figure}\centering
\includegraphics[scale=0.3]{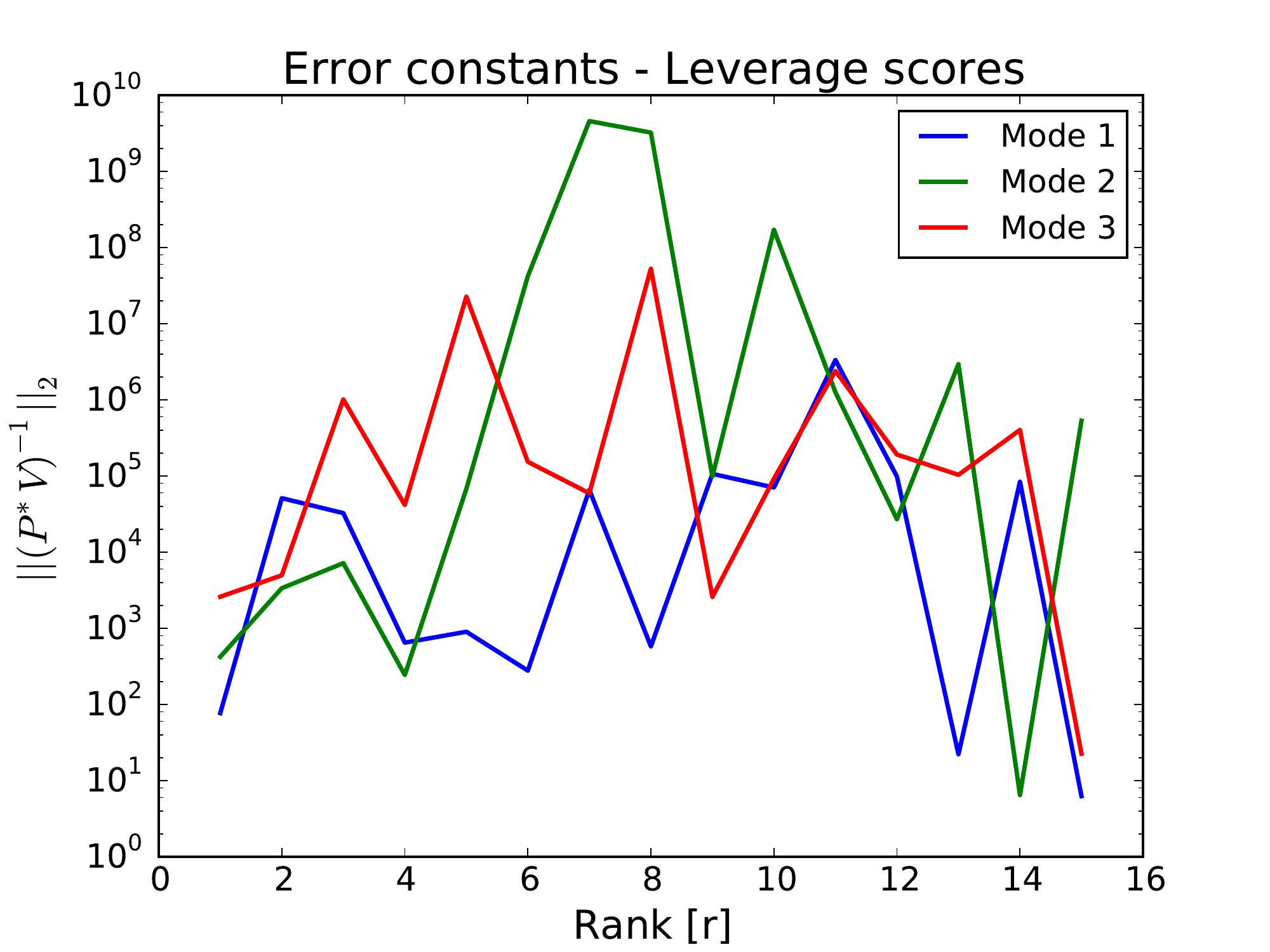}
\includegraphics[scale=0.3]{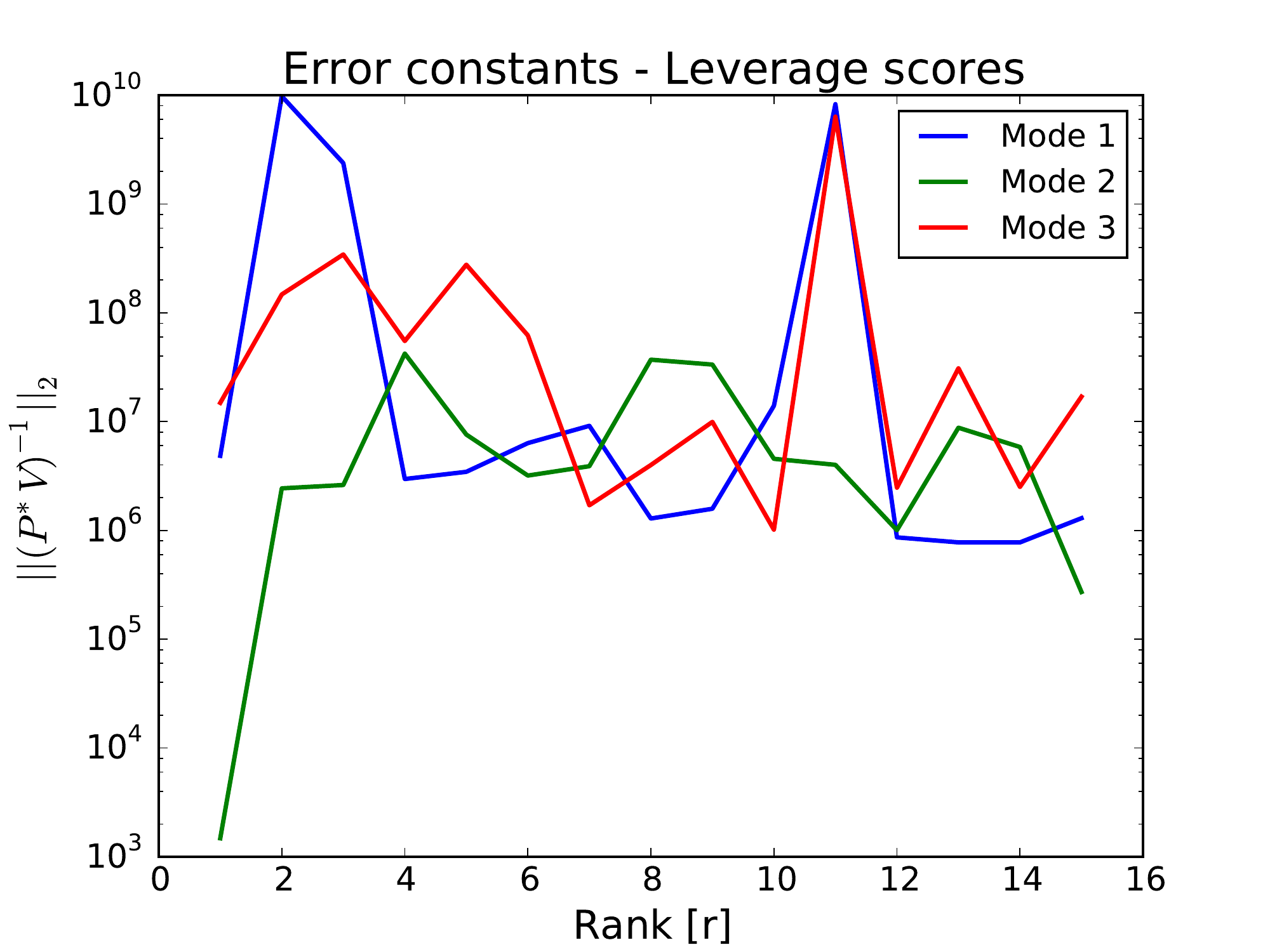}
\caption{Error constants $\| (\mat{P}^*_n\mat{V}_n)^{-1} \|_2$ for each mode computed using the leverage scores (left) using only right singular vectors corresponding to top $15$ singular values, (right) retaining all the right singular vectors.}
\label{fig:err_lev_sparse}
\end{figure}

\subsection{Application: Handwriting classification}
As our first application, we consider the  classification of handwritten digits using the HOSVD, which was investigated by Savas and Eld\'en in~\cite{savas2007handwritten}. Gray-scale images of handwritten digits from $0-9$ are available in $10$ classes, as a training data set; given a new image representing a handwritten digit, the challenge is to assign it a label from $0-9$. In~\cite{savas2007handwritten}, a HOSVD-based algorithm is presented. Here we compare the performance of the HOID and ST-HOID against the HOSVD algorithm.  

For the training and test images, we use the MNIST database which contains $60,000$ training images and $10,000$ test images with $28 \times 28$ pixels in $8$-bit gray-scale. The training images are unequally distributed over the ten classes; to keep the same number over all the digits we restricted the number of training images in every class to $5421$. Another possibility is to duplicate some of the images across the digits. The images are organized into a tensor of size $784 \times 5421 \times 10$ with the first dimension representing the pixels, the second dimension representing the images, and the third dimension representing the digits. 

The classification strategy in~\cite{savas2007handwritten}, relied on the HOSVD. For the classification using the HOID representation, we develop a new strategy here. Let us denote the tensor by $\ten{X}$. In the training phase, we approximate it by a HOID decomposition as $\ten{X} \approx \ten{G} \mode{1} \mat{C}_p \mode{2} \mat{C}_i \mode{3}\mat{C}_d$ where the core tensor is truncated to have size $(62,142,10)$; the sizes of $\mat{C}_p,\mat{C}_i,\mat{C}_d$ are determined appropriately. Next, we compute $\ten{F}  = \ten{G} \mode{1} \mat{C}_p^\dagger \mode{2} \mat{C}_i^\dagger$.  The columns of $\mat{F}^\nu  := \ten{F} (:, :, \nu)$ represent basis vectors for some class given by $\nu$ ranging from $1-10$ representing the digits. Next, an orthonormal basis is computed for each class $\nu$ by retaining the left singular vectors $\mat{U}^\nu$  corresponding to the top $k$ singular values of $\mat{F}^\nu$. In the test phase, given an image $\mat{D}$ we compute $\vec{d} = \text{vec}(\mat{D})$ by an unfolding operation. Next, the image is projected onto the column space of the basis $\mat{C}_p$ by $\vec{d}_p := \mat{C}_p^\dagger \vec{d}$. This low-parameter approximation of $\vec{d}$ is projected orthogonally onto the space spanned by the basis matrix $\mat{U}^\nu$ and its residual is $r(\vec{d}_p,\nu):=\| \left(\mat{I}- \mat{U}^\nu(\mat{U}^\nu)^T\right)\vec{d}_p \|_F$. The image is then assigned a label computed by the minimizer $\argmin_\nu r(\vec{d}_p , \nu)$.

\begin{table}\centering
\begin{tabular}{lccc}
•  & HOSVD & HOID-PQR & ST-HOID \\ \hline
Training time [s] & $123.09$ & $71.48$ & $19.97$ \\
Relative Model error & $0.41$& $0.56$ & $0.54$ \\ \hline
Classification time [s] & $0.67$ & $0.99$ & $0.99$ \\
Basis size $k $ & 15 & $30$ & 30 \\  
Classification Accuracy $\%$ & $95.31$ & $92.01$ & $91.99$ \\\hline 
\end{tabular}•
\caption{Comparison between various algorithms described at the beginning of Section~\ref{sec:res} on the handwriting dataset.  `Training time' refers to the CPU time (measured in seconds) to build the respective low-rank factorizations, `Classification time' refers to projecting the image onto the column space $\mat{C}_p$ and minimizing the residual, to find the closest digit in the database. }
\label{tab:handwriting}
\end{table}•

We compare three different algorithms: HOSVD (which was proposed in~\cite{savas2007handwritten}), HOID algorithm (Algorithm~\ref{alg:hoid}) with Pivoted QR to compute the ID along the modes, and ST-HOID (Algorithm~\ref{alg:sthoid}). Here we choose to use the Pivoted QR algorithm since its performance is empirically comparable to the RRQR. Moreover, PQR has been implemented in LAPACK (xGEPQ3)  and therefore, provides fair timing comparison against the SVD algorithm. From the results in Table~\ref{tab:handwriting}, it can be readily seen that both HOID-PQR and ST-HOID algorithms are considerably faster than the HOSVD algorithm in terms of the training time. The core tensor is truncated to have size $(62,142,10)$, retaining only approximately $0.2\%$ of the entries of the training data set. The size of the basis $k$ used to represent the column space of $\mat{F}^\nu$ was chosen to be $15$ for the HOSVD  based on the numerical experiments in~\cite{savas2007handwritten} and $30$ for both HOID-PQR and ST-HOID based on our numerical experiments. Since a large basis is used, this results in higher classification time for these algorithms. Furthermore, clearly HOSVD has the best classification accuracy; however, the classification accuracy of HOID-PQR and ST-HOID is comparable but are much cheaper to compute. Moreover, the advantage of HOID algorithms is that, it provides a better interpretation over HOSVD.

\subsection{Application: Time-dependent inverse problems}
Motivation: Time dependent inverse problems involve a time dependent parabolic PDE (with appropriate boundary conditions) of the form
\[ \mathcal{S}(\vec{x}; p) \frac{\partial \phi}{\partial t}  - \mathcal{A}(\vec{x}; p)\phi = \delta(\vec{x} - \vec{x}_s) q(t) \qquad \vec{x} \in \Omega, \]
where $\mathcal{S}(\vec{x}; p)$ and $\mathcal{A}(\vec{x}; p)$ are time-independent operators, $\Omega$ is the imaging domain of interest, $p$ is a parameter of interest (possibly infinite dimensional) and $\vec{x}_s$ is the source location. The domain of interest is ``excited'' at several source locations denoted by $\vec{x}_s$ and is represented by the forcing term $\delta(\vec{x} - \vec{x}_s)q(t)$. The response is collected at several measurement locations $\vec{x}_r$ also located in the domain. The measurements, collected at several measurement locations at multiple times due to excitation from multiple sources, are then used to recover the spatial parameters represented by $p$.  This formulation encompasses several well-known imaging modalities -- Diffuse Optical Tomography, Transient Hydraulic Tomography, Electromagnetic inversion, etc. For example, in Transient Hydraulic Tomography $\mathcal{S}(\vec{x};p)$ is the specific storage, and $\mathcal{A}(\vec{x}; p) \define \nabla \cdot(\kappa(\vec{x}; p) \nabla)$, where $\kappa(\vec{x}; p)$ is hydraulic conductivity. The inversion methodology requires discrete measurements of $\phi(\vec{x})$ in space and time to reconstruct specific storage and hydraulic conductivity. The collected data can be expressed as the result of the Green's function $G(\vec{x}_s, \vec{y}_r; t)$ where $\vec{x}_s$ is the source location, $\vec{y}_r$ is the receiver location and $t$ is the time at which the measurement is collected.

\begin{figure}[!ht]\centering
\includegraphics[scale=0.3]{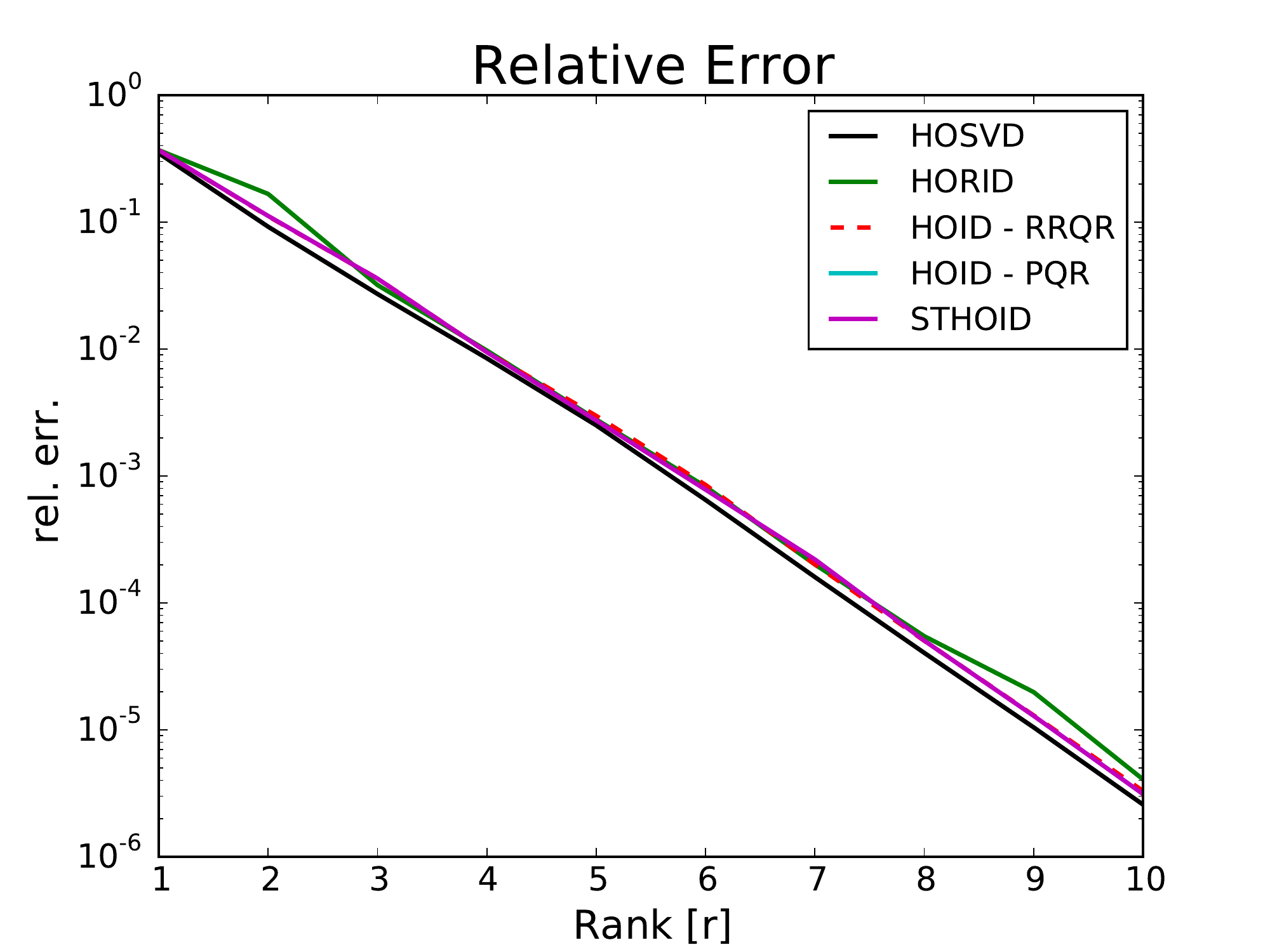}
\includegraphics[scale=0.3]{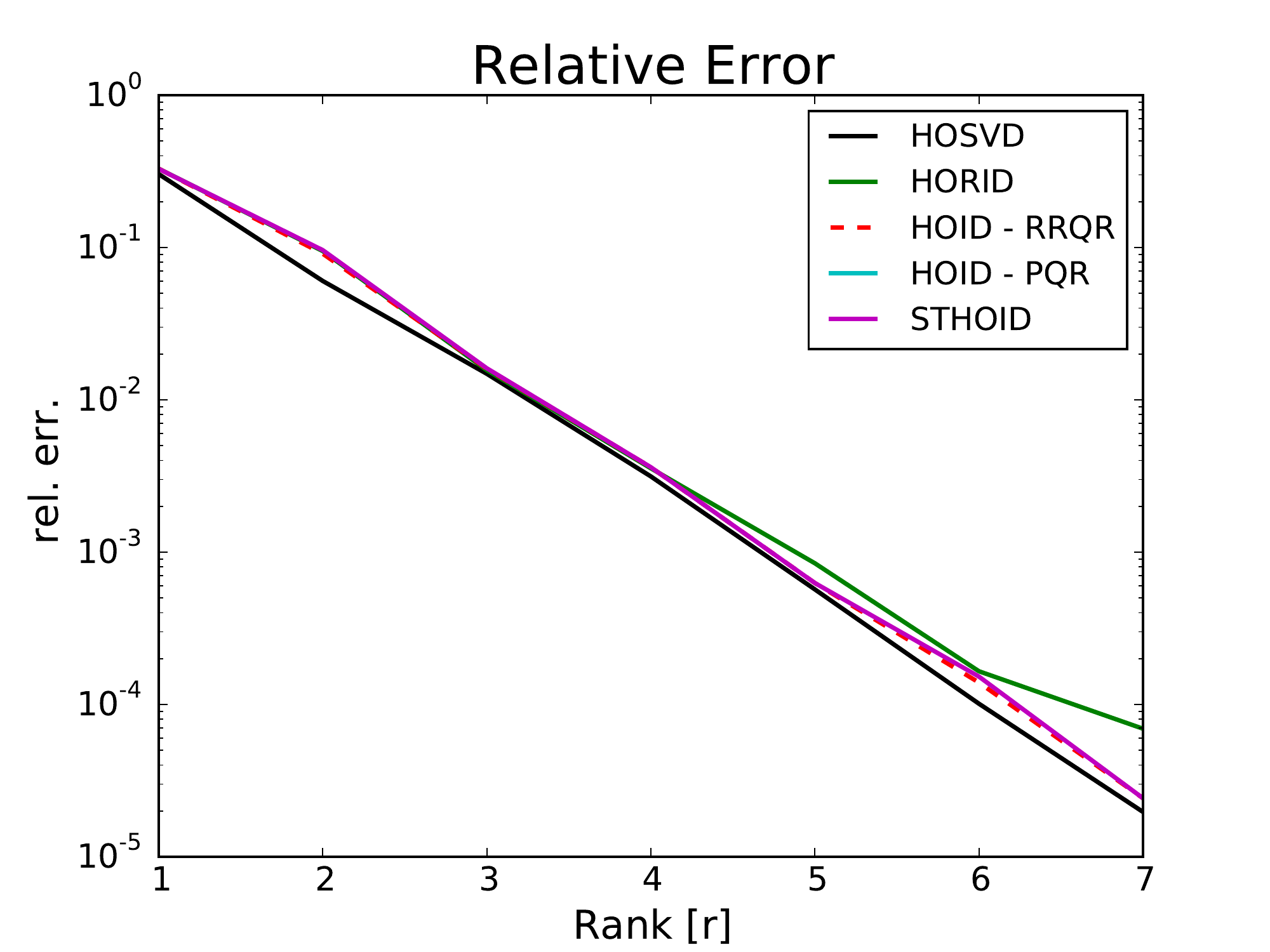}

\caption{ The relative error in the computation of a rank--$(r,\dots,r)$ approximation to the tensors $\ten{X}_3$ (left) and $\ten{X}_5$ (right) defined in~\eqref{eqn:3d5d}. The results are compared for HOSVD, HORID and HOID. As can be seen the approximation errors are comparable for all three algorithms. }
\label{fig:relerr_imaging}
\end{figure}

The data can be viewed as a 3D tensor or a 5D tensor, with respective dimensions
\begin{equation}\label{eqn:3d5d}
 \ten{X}_3 \in \mathbb{R}^{N_s \times N_r \times N_t} \qquad \ten{X}_5 \in \mathbb{R}^{N_{s,x} \times N_{s,y} \times N_{r,x} \times N_{r,y} \times N_t},
\end{equation}• 
where $N_s = N_{s,x} \times N_{s,y}$ is the number of sources (as a product of number of sources in x- and y- directions, respectively), $N_r = N_{r,x} \times N_{r,y}$ is the number of receivers and $N_t$ of time steps at which data is collected at.  When the source and receiver locations are non-overlapping and well-separated, both tensors $\ten{X}_3$ and $\ten{X}_5$ are accurately represented by $(r,r,r)$ or $(r,r,r,r,r)$ dimensional approximation, respectively and the accuracy improves considerably as $r$ increases. Collecting this large dataset can be difficult in many applications, since it requires repeated experiments which are tedious and time-consuming. With the knowledge that the tensor is low-rank, not all entries need be computed and we can use tensor completion algorithms~\cite{liu2013tensor} to acquire all the data. The HOID decomposition is relevant for this application, since it could provide insight into which columns of the mode unfoldings can be sampled which are representative of the data. The columns of the mode unfoldings have the following interpretation - they indicate the appropriate and relevant sources, receivers, or time points. To illustrate the compressibility of the tensor, consider the simple example of the parameters $\mathcal{S} = 1$ and $\mathcal{A} = \Delta$, and $q(t) = 1$ then the Green's function is given by the heat kernel,
\[ G(\vec{x}_s, \vec{y}_r; t) = \frac{1}{(4\pi t )^{d/2}}\exp\left(-\frac{\|\vec{x}_r-\vec{y}_r \|^2}{4t}\right). \]

  The sources are taken to be in the $z = 2$ plane and are evenly spaced between $[-1,1]\times [-1,1]$ and $N_s = 20 \times 20 = 400$. The receivers are co-located with the sources on the plane $z=0$. An image acquisition of the following type is relevant in transmission type geometries, as illustrated by the application~\cite{saibaba2015fast}. The data is also acquired in the time interval $[0.1,1.1]$ with $N_t = 20$ time points. The data can be considered as a 3D tensor or a 5D array both containing $3.2$ million entries. As before we apply HOSVD, HORID and HOID algorithms on both tensors. The results are displayed in Figure~\ref{fig:relerr_imaging}; as can be seen, all three algorithms provide comparable accuracies but HOID and HORID are useful since they provide better interpretations, namely that they indicate which sources, receivers and time points are important. Next, we take an approximation provided by HOSVD, to convert into an equivalent HOID using the DEIM and RRQR strategy for subset selection. The results are viewed in Figure~\ref{fig:deim_imaging_3d} for $\ten{X}_3$ and Figure~\ref{fig:deim_imaging_5d} for $\ten{X}_5$. As can be seen the results of DEIM and RRQR are comparable, but RRQR has the slight edge.

 \begin{figure}[!ht]
\centering
\includegraphics[scale=0.3]{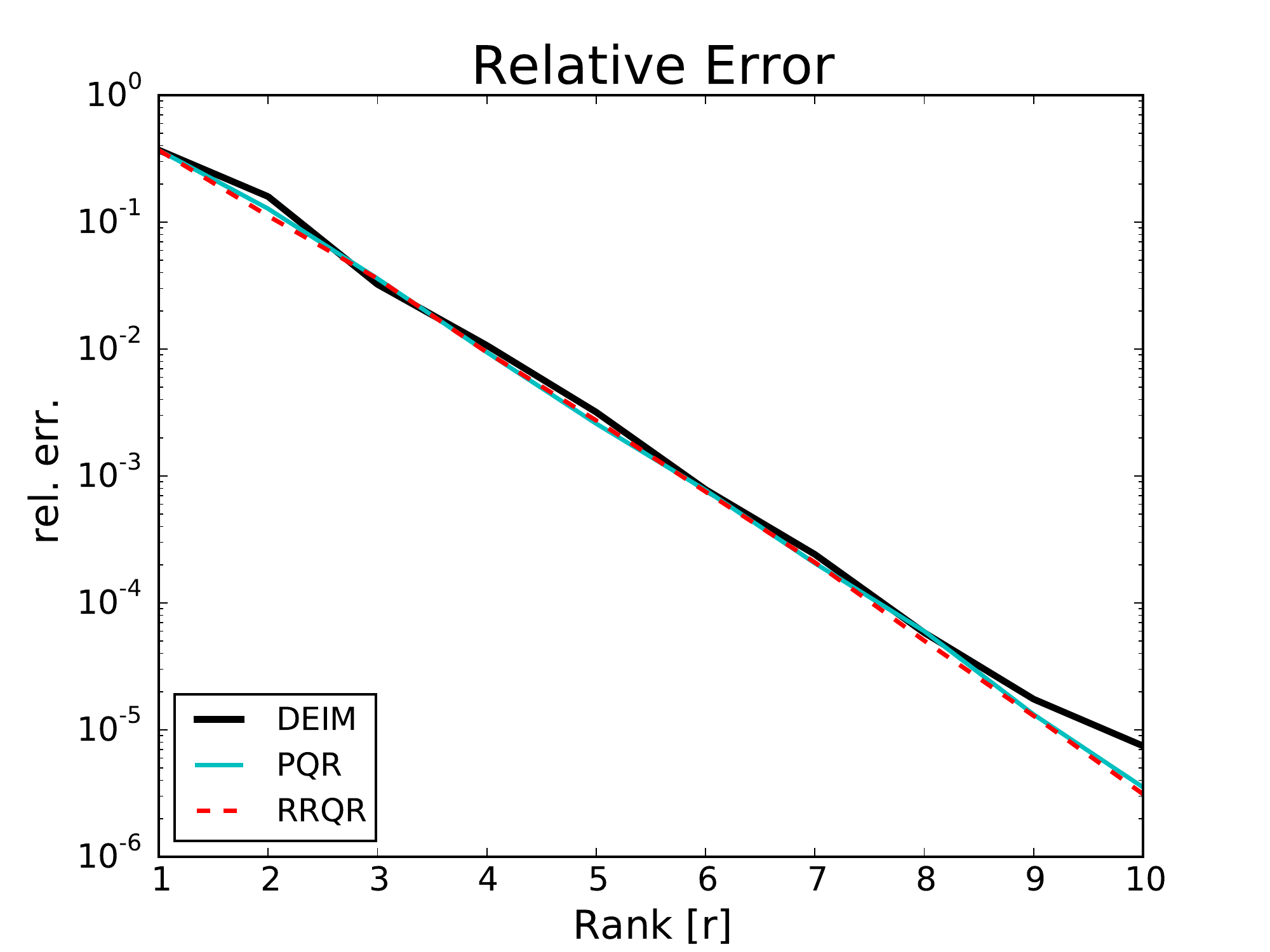}
\includegraphics[scale=0.3]{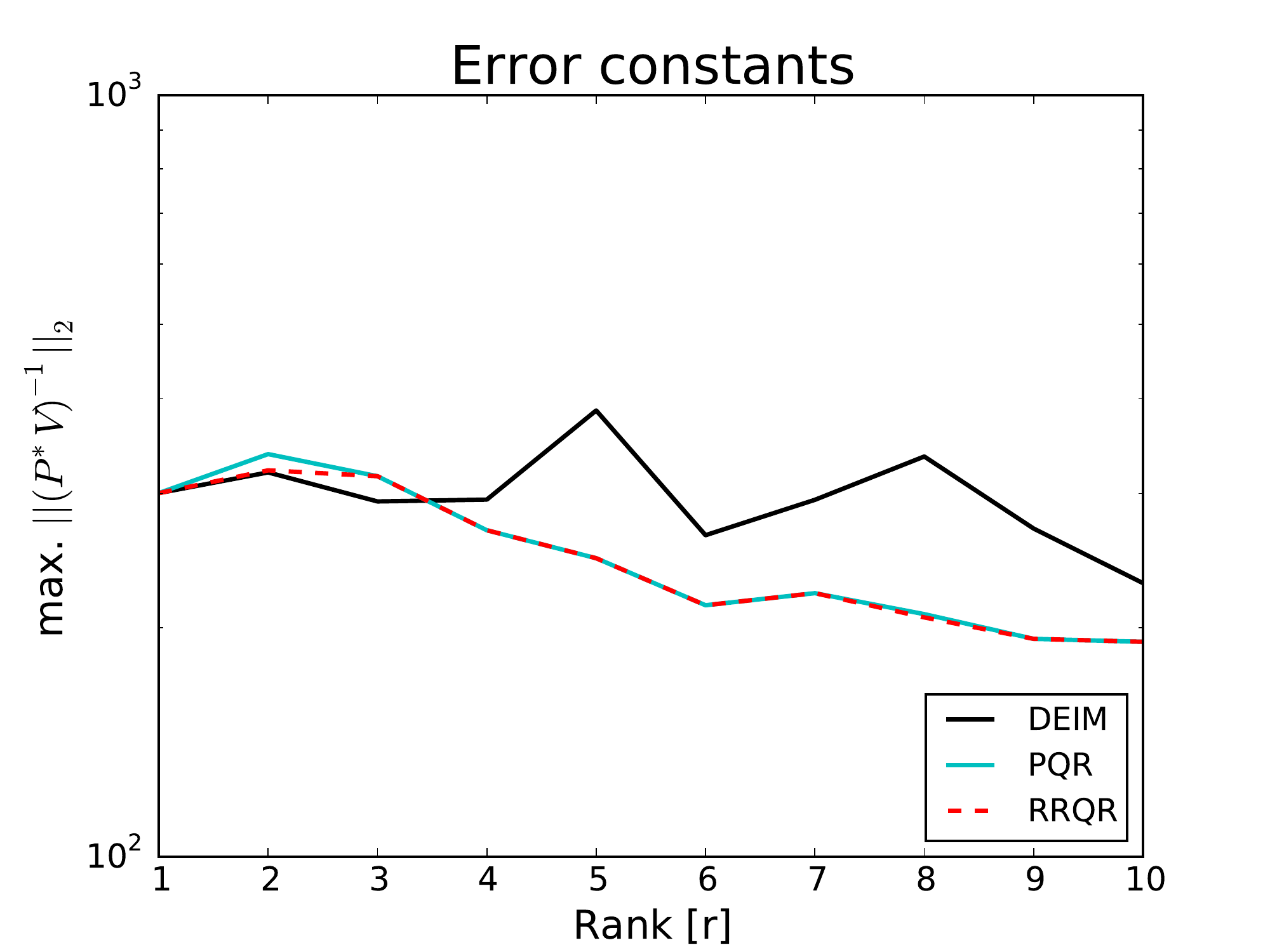}
\caption{(left) The relative error in the computation of a rank-$(r,r,r)$ approximation to tensor $\ten{X}_3$ defined in~\eqref{eqn:3d5d}, starting with the rank-$(r,r,r)$ approximation using  the HOSVD algorithm. The definitions of the algorithms used are provided at the start of Section~\ref{sec:res}. (right) The amplification factor $\max_{n=1,\dots,d}\normf{(\mat{P}_n^*\mat{V}_n)^{-1}}$ computed using DEIM, PQR and RRQR approaches are compared.}
\label{fig:deim_imaging_3d}
\end{figure}•

  \begin{figure}[!ht]
\centering
\includegraphics[scale=0.3]{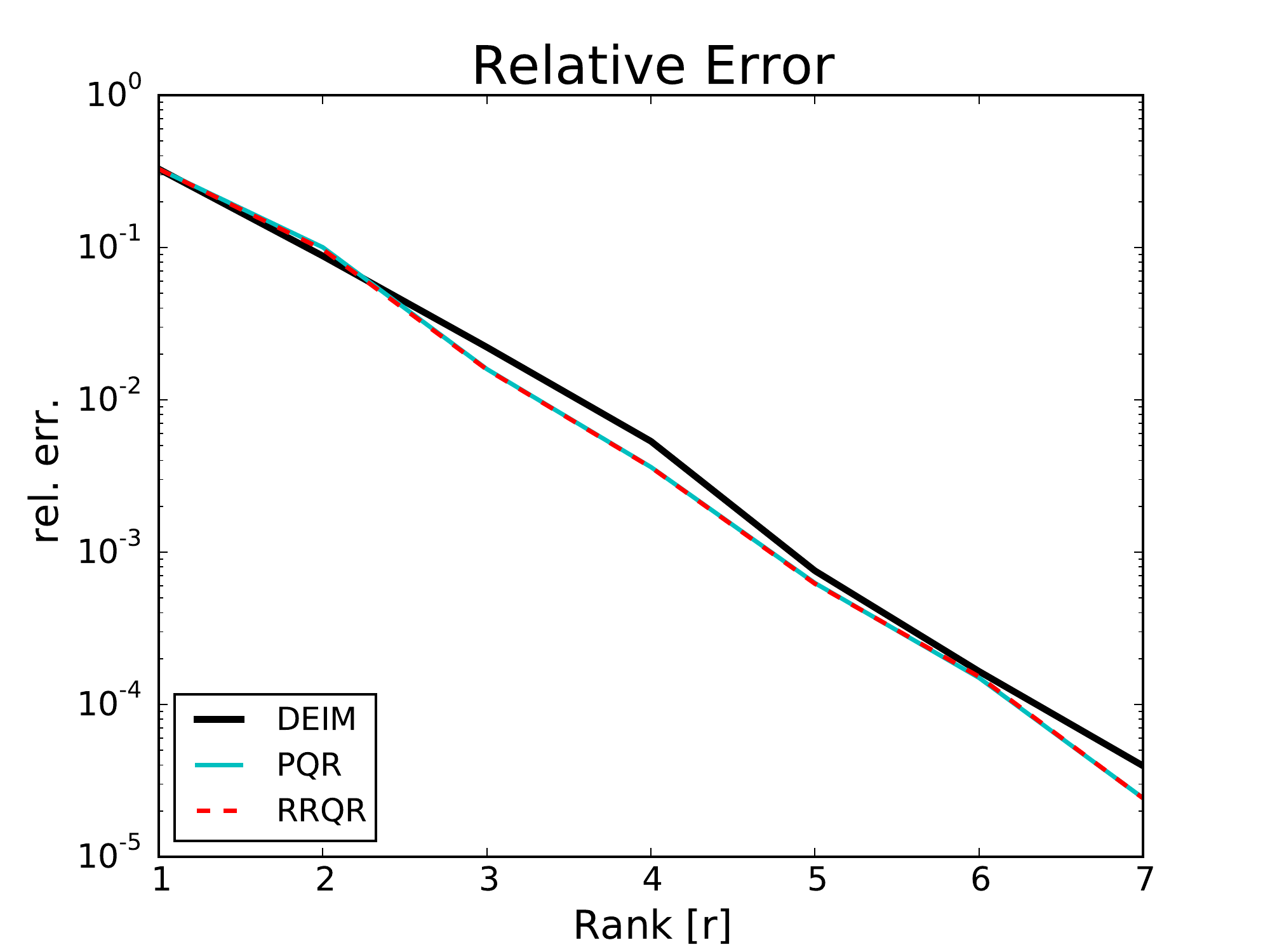}
\includegraphics[scale=0.3]{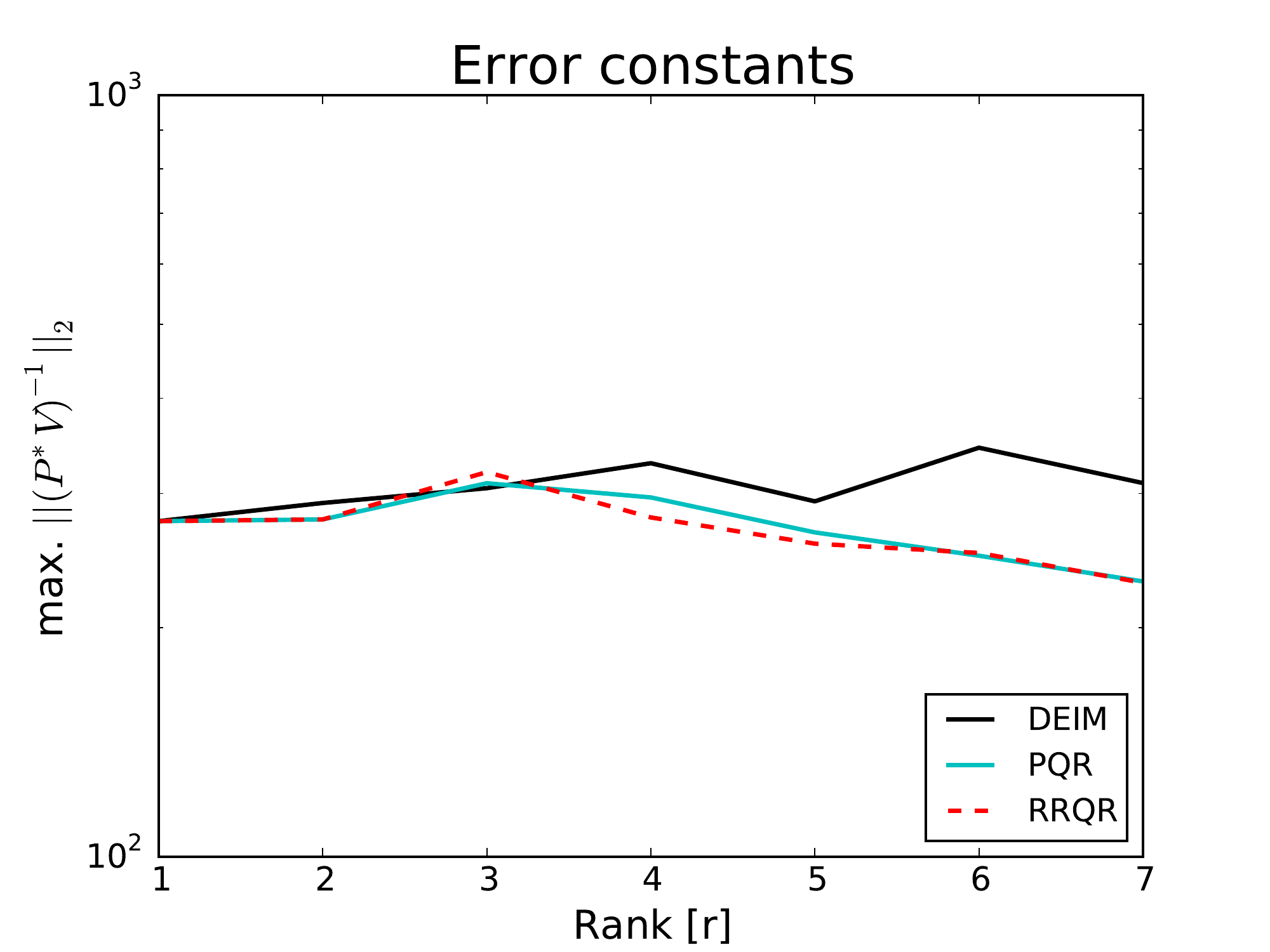}
\caption{(left) The relative error in the computation of a rank$-(r,\dots,r)$ approximation to tensor $\ten{X}_5$ defined in~\eqref{eqn:3d5d}, starting with the rank-$(r,\dots,r)$ approximation using  the HOSVD algorithm. The definitions of the algorithms used are provided at the start of Section~\ref{sec:res}. (right) The amplification factor $\max_{n=1,\dots,d}\normf{(\mat{P}_n^*\mat{V}_n)^{-1}}$ computed using DEIM, PQR and RRQR approaches are compared. }
\label{fig:deim_imaging_5d}
\end{figure}•

 \section{Conclusions}
 We presented Higher Order Interpolatory Decomposition (HOID) - a  CUR-type factorization for tensors in the Tucker format. The algorithms use the strong RRQR to generalize the matrix interpolative decomposition to tensor valued data. For the numerical results presented, the error in the proposed algorithms were comparable to those obtained from HOSVD algorithm. Furthermore, we showed how to convert approximate low-rank representations in the Tucker and CP format, into an equivalent HOID representation. An alternate truncation strategy was also proposed for the HOID decomposition, called ST-HOID. This new algorithm is cheaper to compute and produces acceptable errors, at least for the examples we explored. Numerical comparisons with other algorithms, that required approximate singular vectors, such as DEIM, Pivoted QR and leverage score sampling (Simple-Leverage) were also provided. Although the theoretical results showed that the bounds from RRQR were much lower, in practice,  RRQR, PQR and DEIM gave comparable performance and were better than the Simple-Leverage approach. We also provide an improved analysis of sampling based methods in Appendix~\ref{sec:lev}. The resulting codes have been provided on github~\url{https://github.com/arvindks/tensorcur}. Future work could also look into improving the empirical performance of the leverage score sampling approach. Many of the results developed in this paper for tensors are also applicable to matrix CUR factorizations (particularly, in the Frobenius norm) and will be discussed in an upcoming paper. Other possible extensions could be to apply the interpolatory decomposition to other factorizations such as $\mathcal{H}$-Tucker and Tensor Train formats.

\section{Acknowledgments}
The author is grateful to Ilse Ipsen for helpful discussions which helped shape this paper. She suggested the use of RRQR, which improved the error bounds in the paper significantly. Furthermore, Tania Bakhos gave valuable feedback and in particular, helped improve the  ST-HOID algorithm. The author would also like to thank Ivy Huang for her help and support throughout this process and beyond.

\bibliography{paper_tid}

\begin{thebibliography}{10}

\bibitem{andersson1998improving}
Claus~A Andersson and Rasmus Bro.
\newblock Improving the speed of multi-way algorithms: Part i. {T}ucker3.
\newblock {\em Chemometrics and intelligent laboratory systems}, 42(1):93--103,
  1998.

\bibitem{biagioni2015randomized}
David~J Biagioni, Daniel Beylkin, and Gregory Beylkin.
\newblock Randomized interpolative decomposition of separated representations.
\newblock {\em Journal of Computational Physics}, 281:116--134, 2015.

\bibitem{bodor2012rcur}
Andr{\'a}s Bodor, Istv{\'a}n Csabai, Michael~W Mahoney, and Norbert Solymosi.
\newblock r{CUR}: an {R} package for {CUR} matrix decomposition.
\newblock {\em BMC {B}ioinformatics}, 13(1):103, 2012.

\bibitem{boutsidis2014optimal}
Christos Boutsidis and David~P Woodruff.
\newblock Optimal {CUR} matrix decompositions.
\newblock In {\em Proceedings of the 46th Annual ACM Symposium on Theory of
  Computing}, pages 353--362. ACM, 2014.

\bibitem{broadbent2010subset}
Mary~E Broadbent, Martin Brown, Kevin Penner, I~Ipsen, and R~Rehman.
\newblock Subset selection algorithms: Randomized vs. deterministic.
\newblock {\em SIAM Undergraduate Research Online}, 3:50--71, 2010.

\bibitem{chaturantabut2009discrete}
Saifon Chaturantabut and Danny~C Sorensen.
\newblock Discrete empirical interpolation for nonlinear model reduction.
\newblock In {\em Decision and Control, 2009 held jointly with the 2009 28th
  Chinese Control Conference. CDC/CCC 2009. Proceedings of the 48th IEEE
  Conference on}, pages 4316--4321. IEEE, 2009.

\bibitem{chaturantabut2010nonlinear}
Saifon Chaturantabut and Danny~C Sorensen.
\newblock Nonlinear model reduction via discrete empirical interpolation.
\newblock {\em SIAM Journal on Scientific Computing}, 32(5):2737--2764, 2010.

\bibitem{de2000multilinear}
Lieven De~Lathauwer, Bart De~Moor, and Joos Vandewalle.
\newblock A multilinear singular value decomposition.
\newblock {\em SIAM journal on Matrix Analysis and Applications},
  21(4):1253--1278, 2000.

\bibitem{de2000best}
Lieven De~Lathauwer, Bart De~Moor, and Joos Vandewalle.
\newblock On the best rank-1 and rank-($r_1$, $r_2$,\dots, $r_n$) approximation
  of higher-order tensors.
\newblock {\em SIAM Journal on Matrix Analysis and Applications},
  21(4):1324--1342, 2000.

\bibitem{desilva2008tensor}
Vin De~Silva and Lek-Heng Lim.
\newblock Tensor rank and the ill-posedness of the best low-rank approximation
  problem.
\newblock {\em SIAM Journal on Matrix Analysis and Applications},
  30(3):1084--1127, 2008.

\bibitem{drineas2006fast}
Petros Drineas, Ravi Kannan, and Michael~W Mahoney.
\newblock Fast {M}onte {C}arlo algorithms for matrices iii: Computing a
  compressed approximate matrix decomposition.
\newblock {\em SIAM Journal on Computing}, 36(1):184--206, 2006.

\bibitem{drineas2007randomized}
Petros Drineas and Michael~W Mahoney.
\newblock A randomized algorithm for a tensor-based generalization of the
  singular value decomposition.
\newblock {\em Linear algebra and its applications}, 420(2):553--571, 2007.

\bibitem{drineas2008relative}
Petros Drineas, Michael~W Mahoney, and S~Muthukrishnan.
\newblock Relative-error {CUR} matrix decompositions.
\newblock {\em SIAM Journal on Matrix Analysis and Applications},
  30(2):844--881, 2008.

\bibitem{drmac2015new}
Zlatko Drma\v{c} and Serkan Gugercin.
\newblock A new selection operator for the discrete empirical interpolation
  method--improved a priori error bound and extensions.
\newblock {\em SIAM Journal on Scientific Computing}, 38(2):A631--A648, 2016.

\bibitem{espig2009black}
Mike Espig, Lars Grasedyck, and Wolfgang Hackbusch.
\newblock Black box low tensor-rank approximation using fiber-crosses.
\newblock {\em Constructive approximation}, 30(3):557--597, 2009.

\bibitem{friedland2011fast}
Shmuel Friedland, V~Mehrmann, A~Miedlar, and M~Nkengla.
\newblock Fast low rank approximations of matrices and tensors.
\newblock {\em Electron. J. Linear Algebra}, 22, 2011.

\bibitem{friedland2007generalized}
Shmuel Friedland and Anatoli Torokhti.
\newblock Generalized rank-constrained matrix approximations.
\newblock {\em SIAM Journal on Matrix Analysis and Applications},
  29(2):656--659, 2007.

\bibitem{grasedyck2013literature}
Lars Grasedyck, Daniel Kressner, and Christine Tobler.
\newblock A literature survey of low-rank tensor approximation techniques.
\newblock {\em arXiv preprint arXiv:1302.7121}, 2013.

\bibitem{gu1996efficient}
Ming Gu and Stanley~C Eisenstat.
\newblock Efficient algorithms for computing a strong rank-revealing {QR}
  factorization.
\newblock {\em SIAM Journal on Scientific Computing}, 17(4):848--869, 1996.

\bibitem{halko2011finding}
Nathan Halko, Per-Gunnar Martinsson, and Joel~A Tropp.
\newblock Finding structure with randomness: Probabilistic algorithms for
  constructing approximate matrix decompositions.
\newblock {\em SIAM review}, 53(2):217--288, 2011.

\bibitem{ipsen2015}
Ilse~CF Ipsen.
\newblock Private communication.
\newblock 2015.

\bibitem{ipsen2014sensitivity}
Ilse~CF Ipsen and Thomas Wentworth.
\newblock Sensitivity of leverage scores.
\newblock {\em arXiv preprint arXiv:1402.0957}, 2014.

\bibitem{kolda2009tensor}
Tamara~G Kolda and Brett~W Bader.
\newblock Tensor decompositions and applications.
\newblock {\em SIAM review}, 51(3):455--500, 2009.

\bibitem{liu2013tensor}
Ji~Liu, Przemyslaw Musialski, Peter Wonka, and Jieping Ye.
\newblock Tensor completion for estimating missing values in visual data.
\newblock {\em Pattern Analysis and Machine Intelligence, IEEE Transactions
  on}, 35(1):208--220, 2013.

\bibitem{mahoney2009cur}
Michael~W Mahoney and Petros Drineas.
\newblock {CUR} matrix decompositions for improved data analysis.
\newblock {\em Proceedings of the National Academy of Sciences},
  106(3):697--702, 2009.

\bibitem{mahoney2008tensor}
Michael~W Mahoney, Mauro Maggioni, and Petros Drineas.
\newblock Tensor-{CUR} decompositions for tensor-based data.
\newblock {\em SIAM Journal on Matrix Analysis and Applications},
  30(3):957--987, 2008.

\bibitem{martinsson2014id}
Per-Gunnar Martinsson, Vladimir Rokhlin, Yoel Shkolnisky, and Mark Tygert.
\newblock {ID}: a software package for low-rank approximation of matrices via
  interpolative decompositions.
\newblock Technical report.

\bibitem{oseledets2010tt}
Ivan Oseledets and Eugene Tyrtyshnikov.
\newblock {TT}-cross approximation for multidimensional arrays.
\newblock {\em Linear Algebra and its Applications}, 432(1):70--88, 2010.

\bibitem{oseledets2008tucker}
Ivan~V Oseledets, DV~Savostianov, and Eugene~E Tyrtyshnikov.
\newblock Tucker dimensionality reduction of three-dimensional arrays in linear
  time.
\newblock {\em SIAM Journal on Matrix Analysis and Applications},
  30(3):939--956, 2008.

\bibitem{saibaba2015fast}
Arvind~K. Saibaba, Misha Kilmer, Eric~L. Miller, and Sergio Fantini.
\newblock Fast algorithms for hyperspectral diffuse optical tomography.
\newblock {\em SIAM Journal on Scientific Computing}, 37(5):B712--B743, 2015.

\bibitem{savas2007handwritten}
Berkant Savas and Lars Eld{\'e}n.
\newblock Handwritten digit classification using higher order singular value
  decomposition.
\newblock {\em Pattern recognition}, 40(3):993--1003, 2007.

\bibitem{sorensen2014deim}
Danny~C Sorensen and Mark Embree.
\newblock A {DEIM} induced {CUR} factorization.
\newblock {\em arXiv preprint arXiv:1407.5516}, 2015.

\bibitem{stewart1999four}
GW~Stewart.
\newblock Four algorithms for the the [sic] efficient computation of truncated
  pivoted {QR} approximations to a sparse matrix.
\newblock {\em Numerische Mathematik}, 83(2):313--323, 1999.

\bibitem{vannieuwenhoven2012new}
Nick Vannieuwenhoven, Raf Vandebril, and Karl Meerbergen.
\newblock A new truncation strategy for the higher-order singular value
  decomposition.
\newblock {\em SIAM Journal on Scientific Computing}, 34(2):A1027--A1052, 2012.

\bibitem{voronin2014cur}
Sergey Voronin and Per-Gunnar Martinsson.
\newblock A {CUR} factorization algorithm based on the interpolative
  decomposition.
\newblock {\em arXiv preprint arXiv:1412.8447}, 2014.

\bibitem{wang2013improving}
Shusen Wang and Zhihua Zhang.
\newblock Improving {CUR} matrix decomposition and the {N}ystr{\"o}m
  approximation via adaptive sampling.
\newblock {\em The Journal of Machine Learning Research}, 14(1):2729--2769,
  2013.

\end{thebibliography}
\bibliographystyle{plain}

\appendix
\section{DEIM and improved DEIM}\label{sec:deim}

 Recall, in Algorithm~\ref{alg:hoid2}, Step 3 used the RRQR to generate an index set $\vec{p}_n$ which extracts appropriate columns from $\mat{X}_{(n)}$. We now describe alternative approaches for subset selection. 
In what follows, the DEIM and improved DEIM algorithms are applied to the (approximate) right singular vectors $\mat{V}$ (note that the subscripts have been dropped) from each mode unfolding.

The Discrete Empirical Interpolation method (DEIM) approach was proposed by Chaturantabut and Sorensen~\cite{chaturantabut2010nonlinear} in the context of model reduction of nonlinear dynamical systems. Sorensen and Embree~\cite{sorensen2014deim} adapted this algorithm in the context of subset selection to matrix CUR factorization.

The DEIM selection procedure is summarized in Algorithm~\ref{alg:deim}. It processes the columns of $\mat{V}$ one at a time, to produce the next index. The first index $p_1$ corresponds to the largest magnitude entry of $\vec{v}_1$, denoted by $\norm{\vec{v}_1}{\infty}$. The DEIM selection procedure  has the following error bound from the original analysis~\cite{chaturantabut2010nonlinear}
\begin{equation}\label{eqn:deim1}
\normtwo{(\mat{P}^*\mat{V})^{-1}} \quad \leq \quad \frac{(1 + \sqrt{2n})^{k-1}}{\norm{\vec{v}_1}{\infty}}\,, 
\end{equation}•
which was significantly improved upon by Sorensen and Embree~\cite{sorensen2014deim}
\[ \normtwo{(\mat{P}^*\mat{V})^{-1}} \quad \leq \quad \sqrt{\frac{nk}{3}}2^k. \]
\begin{algorithm}
\begin{algorithmic}[1]
\REQUIRE $\mat{V}$ an $n\times k$ matrix with orthonormal columns
\STATE $\vec{v} = \mat{V}(:,1)$
\STATE $[\sim,p_1] =\max{|\vec{v}|}$; set $\vec{p} = [p_1]$
\FOR {$j=2,\dots,k$}
\STATE $\vec{v} = \mat{V}(:,j)$
\STATE $\vec{c} =\mat{V}(\vec{p},1:j-1)^{-1}\vec{v}(\vec{p})$
\STATE $\vec{r} = \vec{v} - \mat{V}(:,1:j-1)\vec{c}$
\STATE $[\sim,p_j] =\max{|\vec{r}|}$; set $\vec{p} = [\vec{p}; p_j]$
\ENDFOR
\RETURN $\vec{p}$ an integer vector with distinct entries $\{ 1,\dots,n\}$
\end{algorithmic}
\caption{DEIM point selection algorithm~\cite{chaturantabut2009discrete,chaturantabut2010nonlinear}}
\label{alg:deim}
\end{algorithm}

The DEIM procedure can be interpreted as pivoted LU factorization on $\mat{V}^*$~\cite{drmac2015new}. Numerical experiments in~\cite{sorensen2014deim} suggest that although the worst case bounds permits significant growth, this bound is pessimistic for most matrices encountered in practice. Recent work by Drma\v{c} and Gugercin~\cite{drmac2015new} instead uses a pivoted QR algorithm (PQR) applied to $\mat{V}^*$ to  provide the columns from $\mat{X}$. This has been summarized in Algorithm~\ref{alg:deim2} and has the following error bound
\begin{equation}\label{eqn:deim2}
\normtwo{(\mat{P}^*\mat{V})^{-1}} \> \leq \> \sqrt{n - k + 1} \frac{\sqrt{4^k + 6n - 1}}{3}\,, 
\end{equation}•
which is clearly better than the original bound due to DEIM in~\eqref{eqn:deim1} but is comparable with the updated DEIM bound. In our work we provide bounds using the strong RRQR factorization. As mentioned earlier, strong RRQR is more expensive to compute compared to the pivoted QR (or DEIM) but provides much sharper error bounds compared to either algorithms. 

 In practice, however the error incurred using PQR is comparable to strong RRQR (except for adversarial cases) and is comparable to DEIM but better than Simple-Leverage computation. Therefore, PQR is satisfactory in practice because of its relatively low computational costs and high accuracy. 
\begin{algorithm}
\begin{algorithmic}[1]
 \REQUIRE $\mat{V}$ an $n\times k$ matrix with orthonormal columns.
 \STATE Compute the Pivoted QR of $\mat{V}^*$ to obtain $\mat{V}^*\mat{\Pi} = \mat{QR}$.
 \STATE Extract locations of nonzero entries in first $k$ columns of  $\mat{\Pi}$ - call it $\vec{p}$.
\RETURN $\vec{p}$ an integer vector with distinct entries $\{ 1,\dots,n\}$.
\end{algorithmic}
\caption{Improved DEIM point selection algorithm~\cite{drmac2015new}}
\label{alg:deim2}
\end{algorithm}
\section{Leverage score based sampling}\label{sec:lev}
The approach taken by Drineas and Mahoney~\cite{drineas2007randomized} was to obtain the column matrices $\mat{C}_j$ for $j=1,\dots,d$ by choosing columns sampled from a distribution based on the column norms. To obtain a rank$-(r_1,\dots,r_n)$ approximation, they choose as many as $c_j \geq 4\eta^2r_i/ \beta\epsilon^2$ columns from each mode, where $0 < \delta < 1$ and $\beta \in (0,1], $ and $\eta = 1 + \sqrt{(8/\beta)\log(1/\delta)}$. With  probability of failure at most $d\delta$ the error incurred is
 \begin{equation}\label{eqn:drineas2007}\normf{\ten{E}}^2 = \normf{\ten{X} - \ten{G} \mode{1} \mat{C}_1\dots \mode{d}\mat{C}_d }^2\>  \leq \> \sum_{n=1}^d \left(\sum_{k>r_n}\sigma_{k}^2(\mat{X}_{(n)})\right) + d\epsilon \normf{\ten{X}}^2. \end{equation}
This result assumes that the columns are sampled uniformly (without replacement) from a discrete distribution that uses the column norms of the respective mode-$n$ unfoldings. This result can be improved upon by employing multiple passes through the columns~\cite{drineas2007randomized}, but we will not explore this further. Note that the result obtained here is sharper than that obtained in~\cite[Theorem 1]{drineas2007randomized} because we have used the orthogonality of the projectors in Lemma~\ref{lemma:proj}, rather than the triangle inequality.

The authors of~\cite{drineas2007randomized} also note that since the time of initial submission, significant advances have been made that obtain relative error guarantees, as opposed to additive guarantees as obtained from~\eqref{eqn:drineas2007}. These bounds can be improved upon, by using, for example, the leverage scores approach. We now reproduce a result from Mahoney et al~\cite{mahoney2009cur}. For matrices $\mat{A}$, it can be shown that applying Algorithm~\ref{alg:leverage} to the right singular vectors $\mat{V}$ of $\mat{A}$, gives us an approximation of the form
\begin{equation}\label{eqn:leverror}
\normf{\mat{A}-\mat{\Pi}_C \mat{A} } \>\leq\> (1 + \epsilon/2)\normf{\mat{A}-\mat{A}_k},
\end{equation}•
\noindent
which holds with probability at least $99\%$. Here $\mat{\Pi}_C$ is the projection operator into the span of the columns $\mat{C}$. Furthermore, $\mat{A}_k$ is the best rank-$k$ approximation to the matrix $\mat{A}$ obtained by retaining the $k$ left and right singular vectors corresponding to the top $k$ singular values. The sampling strategy used here is based on the Simple-Leverage approach described below. Better results can be obtained using a more sophisticated sampling scheme~\cite{boutsidis2014optimal} but will not be considered here. 
\begin{algorithm}
\begin{algorithmic}[1]
 \REQUIRE $\mat{V}\in \mathbb{C}^{n\times k}$ with orthonormal columns, an accuracy parameter $\epsilon > 0$. 
\STATE Compute normalized leverage scores 
\[ \pi_j = \frac{1}{k}\sum_{l=1}^k \mat{V}_{j,l}^2 . \]
\STATE Keep the $j$-th column of $\mat{A}$ with probability $p_j = \min\{1,c\pi_j\}$ for all $j \in \{1,\dots,m\}$ where $c = \bigO(k\log k/\epsilon^2)$.
\STATE Return the matrix $\mat{C}$ consisting of the selected columns of $\mat{A}$.
\end{algorithmic}
\caption{Column selection using leverage scores.}
\label{alg:leverage}
\end{algorithm}

We now extend the approach in Algorithm~\ref{alg:leverage} developed for matrices to the tensor case by following a similar strategy as~\cite{drineas2007randomized}. Algorithm~\ref{alg:leverage} is applied to each mode unfolding $\mat{X}_n$ after computing the right singular vectors $\mat{V}_n$ (or an approximation to it). Given an accuracy $\epsilon$, the accuracy of the leverage score sampling can be represented using the relative error 
\[  \normf{\ten{E}}^2 = \normf{\ten{X} - \ten{G} \mode{1} \mat{C}_1\dots \mode{d}\mat{C}_d }^2\>  \leq \> (1 + \epsilon/2)^2\sum_{n=1}^d \left(\sum_{k>r_n}\sigma_{k}^2(\mat{X}_{(n)})\right), \]
with probability of success greater than $1-d/10$. The proof is a straightforward application of the results of Lemma~\ref{lemma:proj} and the result in~\eqref{eqn:leverror}. This result is a significant improvement over the results in~\cite{drineas2007randomized}, since the bound depends only the singular values discarded and does not include the additive term $\normf{\ten{X}}$.

We now remark about the practical aspects of the algorithm. The algorithm described above requires us to sample $r_n\log(r_n)/\epsilon^2$ columns from the mode-$n$ unfolding. However, we need only a rank-$r_n$ approximation of each mode-$n$ unfolding. In practice we follow a randomized-deterministic hybrid strategy based on~\cite{boutsidis2014optimal}.  Based on the numerical results in the approach by~\cite{broadbent2010subset} we take only $\max\{4r_n,r_n\log r_n\}$ samples for each mode-$n$ unfolding, sampled from a probability distribution based on the leverage scores (see Algorithm~\ref{alg:leverage}). Subsequently, a strong RRQR is applied on the sampled columns to obtain the final $r_n$ columns. We label this approach as Simple-Leverage. 
\end{document}